\documentclass[a4paper]{amsart}
\usepackage[english]{babel}
\usepackage[T1]{fontenc}
\usepackage[latin1]{inputenc}
\usepackage{amssymb}
\usepackage{amsmath}
\usepackage{amsthm}
\usepackage{amscd}
\usepackage{mathrsfs}
\usepackage{bm}
\usepackage{tikz-cd}
\usepackage{pgf,tikz}
\usepackage{mathrsfs}
\usetikzlibrary{arrows}
\usepackage{pgfplots}
\pgfplotsset{compat=1.15}
\usepackage{float}
\usepackage{verbatim}
\usepackage[all]{xy}
\usepackage[cal=boondoxo]{mathalfa}
\usepackage[numbers]{natbib}         
\usepackage[colorlinks]{hyperref}    
\usepackage{color}                   
\usepackage{hyperref}
\hypersetup{
    colorlinks=true,          
    linktoc=all,              
    linkcolor=blue,           
}

\DeclareMathOperator{\id}{id}

\DeclareMathOperator{\Stab}{Stab}
\DeclareMathOperator{\Bij}{Bij}

\DeclareMathOperator{\Sh}{Sh}
\DeclareMathOperator{\sgn}{sgn}
\DeclareMathOperator{\Sq}{Sq}
\DeclareMathOperator{\tr}{tr}
\DeclareMathOperator{\Gl}{Gl}
\DeclareMathOperator{\Sym}{Sym}

\DeclareMathOperator{\Bgroup}{B}

\DeclareMathOperator{\res}{res}
\DeclareMathOperator{\Ann}{Ann}
\DeclareMathOperator{\lm}{lm}
\DeclareMathOperator{\Quot}{Quot}
\DeclareMathOperator{\Gal}{Gal}
\DeclareMathOperator{\FN}{FN}

\newtheorem{theorem}{Theorem}[section]
\newtheorem{proposition}[theorem]{Proposition}
\newtheorem{corollary}[theorem]{Corollary}
\newtheorem{lemma}[theorem]{Lemma}
\newtheorem*{theorem*}{Theorem}

\theoremstyle{definition}
\newtheorem{definition}[theorem]{Definition}

\theoremstyle{remark}
\newtheorem{remark}[theorem]{Remark}

\newcommand{\Bgrouppos}[1]{\ensuremath \Bgroup_{#1}^{+}}

\newcommand{\dip}[1]{^{ [{#1}] }}

\newcommand{\AB}{A'_{\Bgrouppos{}}}

\begin{document}
	
\author[L. Guerra]{Lorenzo Guerra }
\address{Universit\`a di Roma Tor Vergata}
\email{guerra@mat.uniroma2.it}
\author[S. Jana]{Santanil Jana}
\address{Mathematics Department, 
	University of British Columbia}
\email{santanil@math.ubc.ca}

\begin{abstract}
We show that the direct sum of the cohomology groups of the alternating subgroups of the family of Coxeter groups of Type $ B $ exhibits an almost-Hopf ring structure.
We apply techniques developed by Giusti and Sinha to fully compute a presentation of this structure for mod $ 2 $ coefficient.
\end{abstract}
\thanks{The first author acknowledges the 
	MIUR Excellence Department Project awarded to the Department of Mathematics, University of Rome Tor Vergata, CUP E83C18000100006. 
} 

\title[Mod $ 2 $ cohomology of $ B_n^+ $]{The mod $ 2 $ cohomology rings of the alternating subgroups of the Coxeter groups of Type $ \Bgroup $}
\maketitle

\section{Introduction}

Let $ \Bgroup_n $ be the Coxeter group associated to the following Dinkin diagram:
\begin{center}
\begin{tikzpicture}[line cap=round,line join=round,>=triangle 45,x=1cm,y=1cm]
	\clip(-3.197777777777781,-1) rectangle (7.575555555555557,1);
	\draw [line width=1pt,color=black] (0,0)-- (1,0);
	\draw [line width=1pt,color=black] (1,0)-- (2,0);
	\draw [line width=1pt,color=black] (2,0)-- (2.62,0);
	\draw [line width=1pt,color=black] (3.24,0)-- (4,0);
	\begin{scriptsize}
		\draw [fill=black] (0,0) circle (2.5pt);
		\draw[color=black] (0.07333333333333134,0.19) node {$s_0$};
		\draw [fill=black] (1,0) circle (2.5pt);
		\draw[color=black] (1.0688888888888872,0.19) node {$s_1$};
		\draw [fill=black] (2,0) circle (2.5pt);
		\draw[color=black] (2.0733333333333324,0.19) node {$s_1$};
		\draw [fill=black] (4,0) circle (2.5pt);
		\draw[color=black] (4.073333333333333,0.19) node {$s_{n-1}$};
		\draw[color=black] (0.5088888888888871,0.18111111111111095) node {$4$};
		\draw[color=black] (2.917777777777777,0.06555555555555548) node {$\dots$};
	\end{scriptsize}
\end{tikzpicture}
\end{center}
For the correspondence between Coxeter groups and Dinkin diagrams and, more generally, for the algebra and combinatorics of these objects, we refer to the classical book by Humphreys \cite{Humphreys}.

$ \Bgroup_n $ can be realized as a reflection group in $ \mathbb{R}^n $, by letting $ s_0 $ be the reflection across the coordinate plane $ H_0 = \{ (x_1,\dots,x_n) \in \mathbb{R}^n: x_1 \not= 0 \} $ and, for $ 1 \leq i < n $, $ s_i $ be the reflection across $ H_i = \{(x_1,\dots,x_n) \in \mathbb{R}^n: x_i = x_{i+1}\} $. This makes $ \Bgroup_n $ the group of isometries of the hypercube $ [-1,1]^n $. Alternatively, $ \Bgroup_n $ can be described as the group of signed permutation
\[
\{ \sigma \in \Bij(\{-n,\dots,-1,1,\dots,n\}): \forall 1 \leq i \leq n: \sigma(-i) = -\sigma(i)\},
\]
or as a wreath product $ C_2 \int \Sigma_n = C_2^n \rtimes \Sigma_n $ of a cyclic group of order $ 2 $ and the symmetric group $ \Sigma_n $ on $ n $ objects. $ \Sigma_n $ acts on $ [-1,1]^n $ by permuting the coordinates of the ambient space $ \mathbb{R}^n $, while $ C_2^n $ is the elementary abelian $ 2 $-subgroup generated by reflection across the coordinate hyperplanes $ \{ (x_1,\dots,x_n): x_i = 0 \} $ for $ 1 \leq i \leq n $.

There is a sign homomorphism $ \sgn_{\Bgroup_n} \colon \Bgroup_n \to C_2 = \{-1,1\} $ whose value on every reflection is $ -1 $. The alternating subgroup of $ \Bgroup_n $ is $ \Bgrouppos{n} = \ker(\sgn_{\Bgroup_n}) \subseteq \Bgroup_n $. This construction of the alternating subgroup $ W^+ $ can be performed for every Coxeter group $ W $.
The combinatorics of these groups is relatively well-understood; in this regard, we refer to an article by Brenti--Reiner--Roichman \cite{Brenti-Reiner-Roichman} and the references therein.
In contrast, the cohomology of such groups has not received much attention. Although some general properties are known (for instance, \cite{Akita:17}), a complete calculation of the cohomology ring has been performed only for the classical alternating groups $ \mathcal{A}_n $ (Dinkin diagram of Type $ A_{n-1} $) by Giusti--Sinha \cite{Sinha:17}.

The goal of this paper is to determine the mod $ 2 $ cohomology ring of $ \Bgrouppos{n} $ based on the aforementioned result. Our main motivation is its key role in the description of the cohomology of complete unordered flag manifolds over $ \mathbb{R} $ obtained by the authors in \cite{Guerra-Santanil}. Our main result is Theorem \ref{thm:presentation AB}.

To achieve our goal, we adopt a method introduced in \cite{Sinha:17}, that relies on the interplay between an involution on $ H^*(\Bgrouppos{n}; \mathbb{F}_2) $, a structure similar to a Hopf ring with a coproduct and two products, and a Gysin sequence linking $ H^*(\Bgrouppos{n}; \mathbb{F}_2) $ and $ H^*(\Bgroup_n; \mathbb{F}_2) $. We heavily rely on the arguments in the work cited above throughout the whole paper.

To define our generating classes, we also rely on a geometric and combinatorial tool. There is a resolution of $ \Bgrouppos{n} $ which is a variant of a classical construction by Fox--Neuwirth \cite{Fox-Neuwirth} and we explicitly construct representing cocycles similarly to \cite{Sinha:17}. For the Coxeter groups of Type $ B $, chain level lifting of the coproduct and the second product has been determined by Guerra \cite{Guerra:21}. Building on his calculations, we are able to obtain relations involving these morphisms on our generators.

Another key ingredient is the analysis of the elementary abelian $ 2 $-subgroups of $ \Bgrouppos{n} $ and the invariant subalgebras of their cohomology with respect to their normalizer. A similar study has been performed by Swenson in his thesis \cite{Swenson} for finite Coxeter groups. In his case, the Quillen map is injective and this allows the author to compute the mod $ 2 $ cohomology rings of these reflection groups.
As far as the alternating subgroups $ \Bgrouppos{n} $ are concerned, to the authors' knowledge a similar result was not known. Building on Swenson's work we classify all the maximal elementary abelian $ 2 $-subgroups of $ \Bgrouppos{n} $ up to conjugation. Moreover, inspired by Giusti and Sinha's ideas, we obtain a cohomological detection theorem for this family of subgroups.

Building on these two technical tools, we are able to find a complete set of relations and provide a presentation of our main object. $ \bigoplus_n H^*(\Bgrouppos{n}; \mathbb{F}_2) $ contains a relevant piece of $ \bigoplus_n H^*(\Bgroup_n; \mathbb{F}_2) $ as a sub-algebra with components (with cup product). Since the multiplication in this subalgebra is involved, we decided to only provide a presentation with generators and relations of $ \bigoplus_n H^*(\Bgrouppos{n}; \mathbb{F}_2) $ as an object over a Hopf ring. With this trick, we can unload most of the combinatorial complexity onto $ \bigoplus_n H^*(\Bgroup_n; \mathbb{F}_2) $, which has a complete description explained in Guerra's paper mentioned above.

Apart from this introduction, the organization of the paper is the following:
\begin{itemize}
\item Section 2 is devoted to recalling known facts about the mod $ 2 $ cohomology rings of $ \Bgroup_n $. Their direct sum over $ n $ is not only an algebra with components but is also endowed with a richer Hopf ring structure, that extends and incorporates the cup product. The additional structure can be effectively used as a tool to organize the cup product and obtain ring presentations.
\item In Section 3, following Giusti and Sinha, we describe the algebraic structure of $ \bigoplus_n H^*(\Bgrouppos{n}; \mathbb{F}_2) $. We first define a standard involution on this bigraded object (Subsection 3.1) and establish restriction and transfer maps that link it to the cohomology of $ \Bgroup_n $ (Subsection 3.2). We then proceed to compute a bihomogeneous additive basis $ \mathcal{M}_{charged} $ for $ \bigoplus_n H^*(\Bgrouppos{n}; \mathbb{F}_2) $ via a Gysin sequence, that depends on a choice of certain charged classes. Subsections 3.4 and 3.5 are devoted to defining the almost-Hopf ring structure (with coproduct, cup product, and transfer product) and extending it to a slightly larger object, respectively.
\item In section 4, we develop our geometro-combinatorial tool. We recall Fox-Neuwirth complexes (Subsection 4.1) and explicitly link those for $ \mathcal{A}_n $ and $ \Bgrouppos{n} $ (Subsection 4.2). With this basic setup, we can then construct cocycles representing our generating classes (Subsection 4.3), and subsequently compute their transfer product (Subsection 4.4) and coproduct (Subsection 4.5). The calculations of the coproduct are particularly involved because of the combinatorial complexity of its cochain lifting. We conclude this section with the description of a standard choice of charges that fixes the basis $ \mathcal{M}_{charged} $ once and for all.
\item The contents of Section 5 are the combinatorial description of elementary abelian $ 2 $-subgroups of $ \Bgrouppos{n} $ and the algebraic computation of their cohomology. Precisely, in Subsection 5.1 we recall known results about elementary abelian $ 2 $-subgroups of $ \Bgroup_n $ from Swenson's thesis. In Subsection 5.2, we build on these to completely classify maximal elementary abelian $ 2 $-subgroups in $ \Bgrouppos{n} $ up to conjugation. The goal of Subsection 5.3 is the calculation of the invariant subalgebra of these subgroups with respect to the action of their normalizer. We achieve it by applying Galois theory and other standard techniques in invariant theory. Then we relate these algebras to the almost-Hopf ring structure (Subsection 5.4) by computing restriction to elementary abelian $ 2 $-subgroup of coproducts, transfer products, and cup products. In Subsection 5.5 we determine the restriction of our generators to the cohomology of these subgroups, and we finally prove the Detection Theorem in Subsection 5.6.
\item Section 6 is devoted to the proof of Theorem \ref{thm:presentation AB}. In Subsection 6.1 we first deduce the remaining relations by comparing restrictions to elementary abelian $ 2 $-subgroups and applying our newly proved Detection Theorem, while in Subsection 6.2 we deduce our main result.
\item Finally, the short Section 7 contains a description of the action of the Steenrod algebra on the mod $ 2 $ cohomology of $ \Bgrouppos{n} $.
\end{itemize}

\section{Review of the mod 
\texorpdfstring{$ 2 $}{2} cohomology of 
\texorpdfstring{$ \Bgroup_n $}{Bn}}
From the description of $ \Bgroup_n $ as a wreath product, its mod $ 2 $ cohomology ring can be abstractly deduced from Nakaoka's theorem \cite[Theorem 3.3]{Nakaoka}. However, in practice, it is difficult to describe it in a combinatorially explicit way.

Guerra in \cite{Guerra:21} considers the following linear maps on the bigraded $ \mathbb{F}_2 $-vector space $ A_B = \bigoplus_{n, d} H^d(\Bgroup_n; \mathbb{F}_2) $:
\begin{itemize}
\item the usual cup product, component-by-component $ \cdot A_B \otimes A_B \to A_B $
\item a coproduct $ \Delta \colon A_B \to A_B \otimes A_B $ whose components are induced by the obvious group inclusions $ i_{n,m} \colon B_n \times B_m \to B_{n+m} $
\item a transfer product $ \odot \colon A_B \otimes A_B \to A_B $ whose components are the cohomological transfer maps associated with $ i_{n,m} $
\end{itemize}
Guerra observed that these form a component Hopf ring. However, we observe that this is a particular case of \cite[Theorem 2.37]{Guerra-Salvatore-Sinha} with $ X = \mathbb{P}^\infty(\mathbb{R}) $. We recall here the definition of these objects for convenience.
\begin{definition} \label{def:Hopf ring}
A commutative bigraded component Hopf ring over $ \mathbb{F}_2 $ is a bigraded vector space $ A = \bigoplus_{n,d}A^{n,d} $ with two morphisms $ \odot \colon A \otimes A \to A $ and $ \Delta \colon A \to A \otimes A $ preserving both degrees, a morphism $ \cdot \colon A \otimes A \to A $ preserving the dimension, an augmentation $ \varepsilon \colon A^{0,0} \to \mathbb{F}_2 $ and a unit $ \eta \colon \mathbb{F}_2 \to A^{0,0} $ satisfying the following properties:
\begin{itemize}
\item $ (A,\odot,\Delta,\varepsilon,\eta) $ form a bicommutative bigraded Hopf algebra (with a suitable antipode $ S \colon A \to A $)
\item $ (A,\cdot,\Delta) $ is a bicommutative bialgebra
\item $ A^{n,d} \cdot A^{n',d'} = 0 $ if $ n \not= n' $
\item $ \forall x,y,z \in A $ bihomogeneous, the following distributivity axiom holds:
\[
x \cdot (y \odot z) = \sum (x_{(1)} \cdot y) \odot (x_{(2)} \cdot z)
\]
\end{itemize}
$ A $ is called connected if both $ \varepsilon $ and $ \eta $ are isomorphisms between $ A^{0,*} $ and $ \mathbb{F}_2 $.

Given a bihomogeneous element $ x \in A $, its first degree is denoted $ n(x) $ and called the component of $ x $, its second degree is denoted $ d(x) $ and called the dimension of $ x $.
\end{definition}
In the connected case, a unique antipode always exists and it can be safely neglected.

The publications cited above also provide a complete presentation of this component Hopf ring. It is generated by bihomogeneous classes $ \gamma_{k,m} \in A_B^{m2^k,m(2^k-1)} $ ($k \geq 0, m \geq 1$) and $ w\dip{r} \in A_B^{r,r} $ ($r \geq 1 $). The elements $ \gamma_{k,m} $ arise by restriction from those in the cohomology of the symmetric groups with the same name (see \cite{Sinha:12}). In general $ w\dip{r} $ is obtained from the $ 1 $-dimensional universal Stiefel-Whitney class $ w\dip{1} \in A^{1,1} = H^1(\mathbb{P}^\infty(\mathbb{R}); \mathbb{F}_2) $ by applying certain divided powers operations.

A complete set of relations is described in \cite[Theorem 5.9]{Guerra:21} (where the class $ w\dip{r} $ is denoted $ \delta_r $). For our purpose, we will not use the full presentation and we only recall a bihomogeneous additive basis for $ A_B $. To construct it, we first consider arbitrary cup products of generators with the same component. These are called decorated gathered blocks and will be denoted with the letter $ b $.
The profile of a decorated gathered block is the multiset of the first indices $ k $ of all the factors of the form $ \gamma_{k,\frac{n(b)}{k}} $ with $ k \geq 1 $ appearing in $ b $. The decoration of $ b $ is the cohomology class $ w^a \in H^*(\mathbb{P}(\mathbb{R}); \mathbb{F}_2) $, where $ a $ is the multiplicity of $ w\dip{n(b)} $ as a factor of $ b $.

The coproduct of a decorated gathered block $ b $ is
\[
\Delta(b) = 1 \otimes b + b \otimes 1 + \sum_{(b',b'')} b' \otimes b'', \tag{*}
\]
where the sum is indexed by couples $ (b',b'') $ of decorated gathered blocks having the same profile of $ b $ and such that $ n(b') + n(b'') = n(b) $.
Moreover, if two decorated gathered blocks $ b $ and $ b' $ have the same profile and the same decoration, they satisfy the identity
\[
b \odot b' = \left( \begin{array}{c} n(b)+n(b') \\ n(b) \end{array} \right) b'', \tag{**}
\]
where $ b'' $ is the (necessarily unique) decorated gathered block having the same profile of $ b $ and $ b' $ and component $ n(b'') = n(b) + n(b') $.

A transfer product of decorated gathered blocks $ b_1 \odot \dots \odot b_r $ is called a decorated Hopf monomial if $ b_i $ and $ b_j $ do not have the same profile and the same decoration concurrently for all $ i \not= j $.

\begin{theorem}[from \cite{Guerra:21} or \cite{Guerra-Salvatore-Sinha}] \label{thm:basis B}
The set of decorated Hopf monomials $ \mathcal{M} $ is a bihomogeneous basis of $ A_B $ as a $ \mathbb{F}_2 $-vector space.
\end{theorem}

The cup product of two decorated Hopf monomials can written as a sum of elements of $ \mathcal{M} $ by combining Hopf ring distributivity with the identities $ (*) $ and $ (**) $. A clarifying example can be found in Section 2 of \cite{Guerra-Santanil}. The reader might also be more comfortable with the graphical depiction of the cup product in terms of skyline diagrams discusses in \cite[page 3264]{Guerra:21}.
Although these computations are often lengthy to write, they are elementary and straightforward. For this reason, in the remainder of this article we will often omit their details.

\section{The algebraic structure of 
\texorpdfstring{$ \bigoplus_n H^*(\Bgrouppos{n}; \mathbb{F}_2) $}{the cohomology of Bn} and the additive description}

\subsection{The standard involution and 
\texorpdfstring{$ C_2 $}{Z/2Z}-action}

The main goal of this subsection is to define an involution $ \iota_n \colon H^*(\Bgrouppos{n}; \mathbb{F}_2) \to H^*(\Bgrouppos{n}; \mathbb{F}_2) $ for all $ n $.

\begin{proposition} \label{prop:involution}
	The conjugation $ c_{s} \colon \Bgrouppos{n} \to \Bgrouppos{n} $ by a reflection $ s \in \Bgroup_n $ induces an automorphism $ \iota_n \colon H^*(\Bgrouppos{n}; \mathbb{F}_2) \to H^*(\Bgrouppos{n}; \mathbb{F}_2) $ that does not depend on $ s $.
\end{proposition}
\begin{proof}
	$ \Bgrouppos{n} $ is a normal subgroup of $ \Bgroup_n $, hence $ c_s $ is a well-defined group homomorphism. If $ s_1, s_2 \in \Bgroup_n $ are reflections, then $ c_{s_1} = c_{s_2} c_{s_2s_1} $. Since $ s_2s_1 \in \Bgrouppos{n} $, $ c_{s_2s_1}^* \colon H^*(\Bgrouppos{n}; \mathbb{F}_2) \to H^*(\Bgrouppos{n}; \mathbb{F}_2) $ is the identity. By functoriality, $ c_{s_1} $ and $ c_{s_2} $ induce the same automorphism in cohomology.
\end{proof}
Since every reflection in $ \Bgroup_n $ has order $ 2 $, $ \iota_n^2 = \id_{H^*(\Bgrouppos{n}; \mathbb{F}_2)} $. Note that $ \iota_n $ preserves cohomological dimensions.

We can unify it for all $ n $ by considering all the groups $ \Bgrouppos{n} $ together. This yields the following definition.

\begin{definition} \label{def:AB}
We define $ \AB $ as the bigraded $ \mathbb{F}_2 $-vector space $ \bigoplus_{n,d} H^d(\Bgrouppos{n}; \mathbb{F}_2) $. The standard involution on $ \AB $ is the bigraded linear map $ \iota = \bigoplus_n \iota_n \colon \AB \to \AB $.
\end{definition}

This datum can be rephrased in terms of a $ C_2 $-action on $ \AB $ given by the rule $ (-1).x = \iota(x) $. This makes $ \AB $ a bigraded $ \mathbb{F}_2[C_2] $-module.

\subsection{The connection between 
\texorpdfstring{$ A_B $}{the cohomology of Bn} and \texorpdfstring{$ \AB $}{its alternating subgroup}: restriction and transfer maps}

The restriction $ \res_n \colon H^*(\Bgroup_n; \mathbb{F}_2) \to H^*(\Bgrouppos{n}; \mathbb{F}_2) $ and the transfer homomorphisms $ \tr_n \colon H^*(\Bgrouppos{n}; \mathbb{F}_2) \to H^*(\Bgroup_n; \mathbb{F}_2) $ associated to the inclusions $ \Bgrouppos{n} \hookrightarrow \Bgroup_n $ combine to yield morphisms $ \res = \bigoplus_n \res_n \colon A_B \to \AB $ and $ \tr = \bigoplus_n \tr_n \colon \AB \to A_B $ that relate $ \AB $ to $ A_B $.

\begin{proposition} \label{prop:restriction transfer involution}
The relation of $ \res $ and $ \tr $ with the $ C_2 $-action is the following:
\begin{enumerate}
	\item $ \res $ is $ C_2 $-invariant, i.e., $ \forall x \in A_{B}: \iota(\res(x)) = \res(x) $.
	\item $ \tr $ is $ C_2 $-invariant, i.e., $ \forall x \in \AB: \tr(\iota(x)) = \tr(x) $.
\end{enumerate}	
\end{proposition}
\begin{proof}
Both claims follow formally from the naturality of transfers and restrictions and the commutativity of the following diagram, where $ s \in B_n $ is a reflection:
\begin{center}
\begin{tikzcd}
\Bgrouppos{n} \arrow[hook]{r} \arrow[hook]{rd} & \Bgroup_n \arrow{d}{c_s} \\
& \Bgroup_n
\end{tikzcd}
\end{center}
\end{proof}


\begin{proposition}
$ \tr $ and $ \res $ are related as follows:
\begin{enumerate}
\item $ \forall x \in \AB: \res ( \tr(x) ) = x + \iota(x) $
\item $ \forall y \in A_B: \tr ( \res (y) ) = 0 $
\end{enumerate}
\end{proposition}
\begin{proof}
    Follows from well-known general facts \cite[Lemma II.5.1]{Adem-Milgram}.
\end{proof}

\subsection{The basis of charged decorated Hopf monomials}

In this subsection, we provide a full description of $ \AB $ as a bigraded $ C_2 $-representation using the Gysin sequence of the double covering $ B(\Bgrouppos{n}) \to B(\Bgroup_n) $. We precede the main result with a technical, but elementary, lemma.

\begin{definition} \label{def:Gysin decomposition}
Let $ \mathcal{M} $ be the associated decorated Hopf monomial basis of Theorem \ref{thm:basis B}.
We define $ \mathcal{G}_{ann} $ as the set of those $ x = b_1 \odot \dots \odot b_r \in \mathcal{M} $ whose constituent decorated gathered blocks $ b_i $ all contain at least one factor $ \gamma_{k,l} $ with $ k \geq 2 $.
We also define $ \mathcal{G}_{quot} $ as the set of those $ x = b_1 \odot \dots \odot b_r \in \mathcal{M} $ satisfying one of the following two conditions:
\begin{itemize}
	\item at least one constituent decorated gathered block $ b_i $ of $ x $ is of the form $ (w\dip{n})^k $, and the one with $ k $ maximal (necessarily unique) satisfies $ n \geq 2 $
	\item no $ b_i $ is of the form $ (w\dip{n})^k $ and among the constituent decorated gathered blocks of the form $ \gamma_{1,n}^l (w\dip{2n})^k $ (if there are any), the one with the couple $ (k,l) $ maximal with respect to the lexicographic order $ \leq_{lex} $ on $ \mathbb{N} \times \mathbb{N} $ satisfies $ n \geq 2 $ or $ l = 1 $
\end{itemize}
\end{definition}
For example, $ \gamma_{2,1} \odot (\gamma_{3,2} \cdot w\dip{16}) \in \mathcal{G}_{ann} $, $ \gamma_{2,1} \odot (\gamma_{1,1} \cdot w\dip{2}) \odot (w\dip{2})^5 \odot w^2 \in \mathcal{G}_{quot} \setminus \mathcal{G}_{ann}$, and $ \gamma_{2,1} \odot (\gamma_{1,1} \cdot w\dip{2}) \odot \gamma_{1,2}^2 \notin (\mathcal{G}_{quot} \cup \mathcal{G}_{ann}) $. Note that $ \mathcal{G}_{ann} \subseteq \mathcal{G}_{quot} $.

\begin{lemma} \label{lem:Gysin decomposition}
Let $ e_n = \gamma_{1,1} \odot 1_{n-2} + w\dip{1} \odot 1_{n-1} $. The following statements are true:
\begin{enumerate}
\item $ \mathcal{G}_{ann} $ is a basis of $ \bigoplus_n \Ann_{H^*(\Bgroup_n; \mathbb{F}_2)}(e_n) $
\item $ \mathcal{G}_{quot} $ projects to a basis of $ \bigoplus_n \frac{H^*(\Bgroup_n; \mathbb{F}_2)}{(e_n)} $
\item $ \Ann_{H^*(\Bgroup_n; \mathbb{F}_2)}(e_n) \cap (e_n) = 0 $
\end{enumerate}
\end{lemma}

Before explaining the argument of the proof, we preliminarily need to construct some total orders on some sets of decorated Hopf monomials.

We define a function $ \varphi' \colon \mathcal{B}' \cap A_{\mathbb{P}^\infty(\mathbb{R})}^{n,*} \to \mathbb{N}^{n} $ as follows.
If $ x = b_1 \odot \dots \odot b_r \in \mathcal{B}' $, we order the $ b_i $s in such a way that $ \tilde{\varphi}'(b_1) > \dots > \tilde{\varphi}'(b_r) $.
We let
\[
\varphi'(x) = \left(\underbrace{\tilde{\varphi}'(b_1),\dots,\tilde{\varphi}'(b_1)}_{n(b_1) \mbox{ times}},\underbrace{\tilde{\varphi}'(b_2),\dots,\tilde{\varphi}'(b_2)}_{n(b_2) \mbox{ times}},\dots,\tilde{\varphi}'(b_r)\right).
\]
$ \varphi' $ is injective. Thus, we can define a total order on $ \mathcal{B}' \cap A_{\mathbb{P}^\infty(\mathbb{R})}^{n,*} $ by restricting the $ n $-fold lexicographic order on $ \mathbb{N}^n $ along $ \varphi' $. Provide $ \mathcal{B}' = \sqcup_n (\mathcal{B}' \cap A_{\mathbb{P}^\infty}^{n,*}) $ with the total order that restricts to the one defined above on each component, and such that for all $ x,x' \in \mathcal{B}' $, $ n(x) < n(x') \Rightarrow x < x' $.

We also consider the set $ \mathcal{B}'' $ of those $ x = b_1 \odot \dots \odot b_r \in \mathcal{M} $ such that for all $ 1 \leq i \leq r $: $ b_i = \gamma_{1,n_i}^{l_i}(w\dip{2n_i})^{k_i} $ for some $ n_i,l_i \geq 1 $ and $ k_i \geq 0 $.
We define a function $ \tilde{\varphi}'' $ from the set of decorated gathered blocks in $ \mathcal{B}'' $ to $ \mathbb{N} \times \mathbb{N} $ by letting $ \tilde{\varphi}''(\gamma_{1,n}^l(w\dip{2n})^k) = (k,l) $.
Note that $ \tilde{\varphi}''(b) = \tilde{\varphi}''(b') $ if and only if $ b $ and $ b' $ have the same profile and decoration.
Consequently, if $ x \in \mathcal{B}'' $ is a decorated Hopf monomial, $ \tilde{\varphi}'' $ assumes pairwise distinct values on its constituent gathered blocks.
We define a function $ \varphi'' \colon \mathcal{B}'' \cap A_{\mathbb{P}^\infty(\mathbb{R})}^{2n,*} \to (\mathbb{N} \times \mathbb{N})^n $ as follows. If $ x = b_1 \odot \dots \odot b_r \in \mathcal{B}'' $, we order the $ b_i $s in such a way that $ \tilde{\varphi}''(b_1) > \dots > \tilde{\varphi}''(b_r) $ with respect to the lexicographic order on $ \mathbb{N}\times \mathbb{N} $.
We let
\begin{gather*}
	\varphi''(x) = \left( \underbrace{\tilde{\varphi}''(b_1),\dots,\tilde{\varphi}''(b_1)}_{\frac{n(b_1)}{2} \mbox{ times}},\underbrace{\tilde{\varphi}''(b_2),\dots, \tilde{\varphi}''(b_2)}_{\frac{n(b_2)}{2} \mbox{ times}},\dots,\tilde{\varphi}''(b_r) \right).
\end{gather*}
$ \varphi'' $ is injective. Thus, we can define a total order on $ \mathcal{B}'' \cap A^{2n,*}_{\mathbb{P}^\infty(\mathbb{R})} $ by restricting the $ n $-fold lexicographic power of $ (\mathbb{N} \times \mathbb{N}, \leq_{lex}) $ along $ \varphi'' $.
Provide $ \mathcal{B}'' = \sqcup_n (\mathcal{B}'' \cap A_{\mathbb{P}^\infty}^{2n,*}) $ with the total order that is restricted to the one defined above on each component, and such that for all $ x,x' \in \mathbb{B}'' $, $ n(x) < n(x') \Rightarrow x < x' $.

 We observe that transfer product gives a bijection $ \mathcal{B}' \times \mathcal{B}'' \times \mathcal{G}_{ann} \to \mathcal{B}' \odot \mathcal{B}'' \odot \mathcal{G}_{ann} = \mathcal{M} $. Choose an arbitrary total order on $ \mathcal{G}_{ann} $ and provide $ \mathcal{M} $ with the lexicographic product of the total orders on $ \mathcal{B}' $, $ \mathcal{B}'' $ and $ \mathcal{G}_{ann} $.

\begin{proof}[Proof of Lemma \ref{lem:Gysin decomposition}]
We will prove the statements of the Lemma by expressing a cup product $ e_n \cdot x $ as a linear combination in $ \mathcal{M} $ and analyzing its leading non-zero term $ \lm(e_n \cdot x) $ of with respect to that order on $ \mathcal{M} $.
To simplify notation, given a Hopf monomial $ x $ in $ A_{\mathbb{P}^\infty(\mathbb{R})}^{n,d} $, we write $ x = x' \odot x'' \odot x''' $, with $ x' \in \mathcal{B}' $, $ x'' \in \mathcal{B}'' $ and $ x''' \in \mathcal{G}_{ann} $.

Since $ \Delta(x''') $ has no addends in component $ (1,n(x''')-1) $ or $ (2,n(x''')-2) $, by Hopf ring distributivity $ (w \odot 1_{n(x''')-1}) \cdot x''' = (\gamma_{1,1} \odot 1_{n(x''')-2}) \cdot x''' = 0 $. Therefore, by Hopf ring distributivity again, $ e_n \cdot x = \left((e_{n-n(x''')}) \cdot (x' \odot x'')\right) \cdot x''' $.

In particular, if $ x' = x'' = 1_0 $, then $ x \in \Ann(e_n) $. This proves that $ \mathcal{G}_{ann} \subseteq \Ann(e_n) $.

If $ x' = 1_0 $ and $ x'' \not= 1_0 $, then $ e_n \cdot x = (e_{n(x'')} \cdot x'') \odot x''' $. Moreover, $ \Delta(x'') $ has no addends in component $ (1,n(x'')-1) $, hence $ (w \odot 1_{n(x'')-1}) \cdot x'' = 0 $.
$ (\gamma_{1,1} \odot 1_{n(x'')-2}) \cdot x'' $ is the sum of all Hopf monomials $ y \in \mathcal{B}'' $ such that $ \varphi''(y) $ is obtained from $ \varphi''(x'') $ by adding $ (0,1) $ to one of its $ n $ factors and re-ordering them from the largest to the smallest according to the lexicographic order of $ \mathbb{N} \times \mathbb{N} $. We deduce that $ \lm(e_n \cdot x) = \lm((\gamma_{1,1} \odot 1_{n(x)-2}) \cdot x) $ is equal to
\[
{\varphi''}^{-1}(\varphi''(x'') + ((0,1),(0,0),\dots,(0,0))) \odot x'''.
\]

Similarly, if $ x' \not= 1_0 $, then
\[
\lm((w \odot 1_{n(x)-1}) \cdot x) = {\varphi'}^{-1}(\varphi'(x') + (1,0,\dots,0)) \odot x'' \odot x'''.
\] 
Since $ (\gamma_{1,1} \odot 1_{n(x)-2}) \cdot x $ is a linear combination of strictly smaller Hopf monomials, this is also the leading term of $ e_n \cdot x $.

The leading terms of $ e_n \cdot x $ for $ x \in \mathcal{M} \setminus \mathcal{G}_{ann} $ that we just determined are all distinct, so the matrix associated to $ e_n \colon A_{\mathbb{P}^\infty(\mathbb{R})} \to A_{\mathbb{P}^\infty(\mathbb{R})} $ with respect to the basis $ \mathcal{M} $ of $ A_{\mathbb{P}^\infty(\mathbb{R})} $ with our total order is in echelon form with exactly those leading entries. This implies that $ \mathcal{G}_{ann} $ generates $ \Ann(e_n) $ and proves $ (1) $.
By observing that
\[
\mathcal{G}_{quot} = \mathcal{M} \setminus \{ \lm(e_n \cdot x): x \in \mathcal{M} \setminus \mathcal{G}_{ann} \}
\]
we deduce $ (2) $. $ (3) $ also follows from this fact because none of those leading terms belong to $ \Ann(e_n) $.
\end{proof}

\begin{theorem} \label{thm:basis alternating subgroup}
Let $ \mathcal{M} $ be as in Definition \ref{def:Gysin decomposition}. Let $ \kappa = \mathbb{F}_2 $ be the trivial $ C_2 $-representation and let $ \xi = \mathbb{F}_2[C_2] $ be the regular $ C_2 $-representation. 
\begin{enumerate}
	\item If $ x \in \mathcal{G}_{ann} $, then there exists two classes $ x^+, x^- \in \AB $ with the same bidegree of $ x $ such that $ \iota(x^+) = x^- $, $ \tr(x^+) = \tr(x^-) = x $ and $ \res(x) = x^+ + x^- $.
	\item For $ x \in \mathcal{G}_{quot} \setminus \mathcal{G}_{ann} $, let $ x^0 = \res(x) $. Let $ \mathcal{M}_0 = \{ x^0: x \in \mathcal{G}_{quot} \setminus \mathcal{G}_{ann} \} $, $ \mathcal{M}_+ = \{ x^+: x \in \mathcal{G}_{ann} \} $ and $ \mathcal{M}_- = \{ x^-: x \in \mathcal{G}_{ann} \} $. Then $ \mathcal{M}_{charged} = \mathcal{M}_0 \sqcup \mathcal{M}_+ \sqcup \mathcal{M}_- $ is a basis of $ \AB $.
	\item As a $ C_2 $-representation, $ \AB \cong \bigoplus_{\mathcal{M}_0} \kappa \oplus \bigoplus_{\mathcal{M}_+} \xi $.
\end{enumerate}
\end{theorem}
\begin{proof}
Since the index of $ \Bgrouppos{n} $ in $ \Bgroup_n $ is $ 2 $, $ B(\Bgrouppos{n}) \to B(\Bgroup_n) $ is homotopy equivalent to a double covering. Hence, there is a Gysin exact sequence:
\[
\dots \to H^k(\Bgroup_n; \mathbb{F}_2) \stackrel{\res_n}{\rightarrow} H^k(\Bgrouppos{n}; \mathbb{F}_2) \stackrel{\tr_n}{\rightarrow} H^k(\Bgroup_n; \mathbb{F}_2) \stackrel{\partial}{\rightarrow} H^{k+1}(\Bgroup_n; \mathbb{F}_2) \to \dots 
\]
$ \partial $ is the multiplication by the Euler class of the covering, which is $ e_n $ by \cite[Lemma 3.4]{Guerra:23}.

Therefore, Lemma \ref{lem:Gysin decomposition} guarantees that $ \mathcal{G}_{quot} $ and $ \mathcal{G}_{ann} $ form a Gysin decomposition in the sense of \cite[Definition 3.6]{Sinha:17}.
The result follows from this remark and the discussion of \cite[Subsection 3.3]{Sinha:17}. 
\end{proof}

Since elements of $ \mathcal{G}_{ann} $ have component multiple of $ 4 $, we immediately deduce the following corollary.

\begin{corollary} \label{cor:quotient not mod 4}
If $ n \not\equiv 0 \mod 4 $, then $ \res_n \colon H^*(\Bgroup_{n}; \mathbb{F}_2) \to H^*(\Bgrouppos{n}; \mathbb{F}_2) $ is surjective.
\end{corollary}

\subsection{Almost-Hopf ring structure on \texorpdfstring{$ \AB $}{the cohomology of the alternating subgroup of Bn}}

On $ \AB $ we consider structural morphisms similar to those defined on $ A_B $:
\begin{itemize}
	\item the usual in-component cup product $ \cdot \colon \AB \otimes \AB \to \AB $,
	\item a coproduct $ \Delta \colon \AB \to \AB \otimes \AB $ whose components are restriction maps associated to the inclusions of groups $ \Bgrouppos{n} \times \Bgrouppos{m} \to \Bgrouppos{n+m} $,
	\item a transfer product $ \odot \colon \AB \otimes \AB \to \AB $ whose components are transfer maps associated to the same inclusions.
\end{itemize}

\begin{definition}[from \cite{Sinha:17}] \label{def:Hopf semiring}
	A commutative bigraded component almost-Hopf semiring is a bigraded vector space $ A = \bigoplus_{n,d}A^{n,d} $ with a coassociative and cocommutative coproduct $ \Delta $, two associative and commutative products $ \odot,\cdot $, a unit $ \eta \colon k \to A $, and a counit $ \varepsilon \colon A \to k $ satisfying all the properties of Definition \ref{def:Hopf ring}, except those involving the antipode and $ \Delta $ and $ \odot $ forming a bialgebra.
\end{definition}

\begin{theorem}[{\cite[Theorem 2.4]{Sinha:17}}, particular case] \label{thm:alternating subgroup Hopf semiring}
	Let $ \AB = \bigoplus_{n \geq 0} H^*(\Bgrouppos{n}; \mathbb{F}_2) $. With the maps $ \Delta, \odot, \cdot $ defined above, and the unit and the counit that identify $ {\AB}^{0,0} $ with $ \mathbb{F}_2 $, $ \AB $ is a commutative bigraded component almost-Hopf semiring over $ \mathbb{F}_2 $.
\end{theorem}

\subsection{\texorpdfstring{The extension $ \widetilde{\AB} $}{The extended almost-Hopf ring}}

The standard involution interacts nicely with the structural morphisms of this Hopf semiring.
\begin{proposition} \label{prop:relations involution}
The following identities hold for all $ x,y \in \AB $:
\begin{enumerate}
\item $ \iota(x) \odot y = x \odot \iota(y) = \iota(x \odot y) $.
\item $ (\iota \otimes \id) \Delta(x) = (\id \otimes \iota) \Delta(x) = \Delta \iota(x) $.
\item $ (\iota \otimes \iota) \Delta(x) = \Delta(x) $.
\item $ \iota(x) \cdot \iota(y) = \iota(x \cdot y) $.
\end{enumerate}
\end{proposition}
\begin{proof}
All the statements can be proved diagrammatically, with essentially the same argument used in the proof of \cite[Lemma 5.10]{Guerra:21}. Alternatively, one can apply the chain-level argument of \cite[Proposition 3.17]{Sinha:17}.
\end{proof}

Thanks to Proposition \ref{prop:relations involution}, we can extend our almost Hopf ring similarly to \cite[Definition 3.22, Proposition 3.23]{Sinha:17}.

\begin{definition} \label{def:extended Hopf ring}
Let $ \widetilde{\AB} $ be the bigraded vector space that agree with $ \AB $ in positive component, and such that $ \widetilde{\AB}^{0,*} = \mathbb{F}_2 \{ 1_0^+, 1_0^- \} $ concentrated in bidegree $ (0,0) $. The structural operations $ \cdot $, $ \odot $ and $ \Delta $ are extended to $ \widetilde{\AB} $ by the rules $ 1_0^+ \cdot 1_0^+ = 1_0^+ $, $ 1_0^- \cdot 1_0^- = 1_0^- $, $ 1_0^+ \cdot 1_0^- = 0 $, $ 1_0^+ \odot x = x $, $ 1_0^- \odot x = \iota(x) $, $ \Delta(x) = \overline{\Delta}(x) + 1_0^+ \otimes x + 1_0^- \otimes x + x \otimes 1_0^+ + x \otimes 1_0^- $, where $ \overline{\Delta} $ is the reduced coproduct in $ \AB $.
\end{definition}

\begin{proposition} \label{prop:extended Hopf ring}
$ \widetilde{\AB} $, with the extended morphisms $ \cdot $, $ \odot $ and $ \Delta $, is a bigraded component quasi-Hopf ring containing $ \AB $ as a subobject.
\end{proposition}

The standard involution extends to $ \widetilde{\AB} $ by letting $ \iota(1_0^+) = 1_0^- $ and $ \iota(1_0^-) = 1_0^+ $. Proposition \ref{prop:relations involution} holds true in $ \widetilde{\AB} $.
Moreover, $ \res $ and $ \tr $ extend as well, by the assignment $ \res(1_0) = 1_0^+ + 1_0^- $ and $ \tr(1_0^+) = \tr(1_0^-) = 1_0 $. With a slight abuse of notation, we use the same notation for the extended and non-extended versions of these maps, because there is no risk of confusion.

The relation between $ \res $, $ \tr $ and the Hopf semiring structure of $ \widetilde{\AB} $ is stated explicitly below.
\begin{proposition} \label{prop:identities restriction transfer}
	The following statements are true:
	\begin{enumerate}
		\item $ \res $ preserves coproducts and cup products.
		\item $ \tr $ preserves transfer products.
		\item $ \forall x,y \in A_B: \res(x) \odot \res(y) = 0 $.
		\item $ \forall x \in \AB, y \in A_B: \res(\tr(x) \odot y) = x \odot \res(y) $.
		\item $ \forall x,x' \in \AB: \tr(x) \cdot \tr(x') = \tr(x \cdot x' + \iota(x) \cdot x') $.
		\item $ \forall x \in \AB, y \in A_B: \tr(x \cdot \res(y)) = \tr(x) \cdot y $.
	\end{enumerate}
\end{proposition}
\begin{proof}
    All of these statements can be verified diagrammatically. The first two statements follow from the functoriality of transfer and restriction maps. Statements (3) and (4) are proved similarly as \cite[Proposition 3.16]{Sinha:17}. Statement (5) follows from \cite[Proposition 3.18]{Sinha:17} and the remaining identities can be proved with the same argument used for \cite[Proposition 5.29]{Guerra:21}.
\end{proof}

\section{The geometro-combinatorial picture: Fox-Neuwirth representatives and  \texorpdfstring{generators of $ \AB $}{almost-Hopf ring generators}}

We devote this section to the definition of the generators of $ \AB $ as an almost-Hopf ring. We will construct them by exhibiting cochain representatives in the Fox-Neuwirth complex.

\subsection{Recollection of Fox-Neuwirth cochain complexes}

We briefly recall below specific combinatorial models for the cohomology of the symmetric and hyperoctahedral groups. We refer to \cite{Sinha:13,Guerra:21} for a detailed exposition.

We fix cartesian coordinates $ x $ and $ y $ on a plane and an integer $ m $.
We also consider planar trees $ T $ with the following properties:
\begin{itemize}
\item $ T $ has a root in the origin of the plane
\item the difference between the $ y $ coordinate of two vertices connected by an edge is $ 1 $
\item the leaves all lie on the line described by the equation $ y = m $
\item the $ y $ coordinate is always increasing along every minimal path from the root to a leaf of $ T $
\item $ T $ has $ n $ leaves that are labeled bijectively with the set $ \{1,\dots,n\} $
\end{itemize}
We call such trees planar level trees.

We introduce an equivalence relation to remove the dependence on $ m $. Given a planar level tree $ T $ with $ n $ leaves $ l_1,\dots,l_n $ of height $ m $, we can construct another tree $ T(1) $, with $ n $ leaves $ l'_1, \dots, l'_n $ such that $ x(l'_i) = x(l_i) $ and $ y(l'_i) = y(l_i) + 1 $ by adding the straight edges between $ l_i $ and $ l_i' $. We consider the equivalence relation $ \sim $ generated by the isotopy relation in the plane and $ T \sim T(1) $ for all $ T $.

There is a cochain complex of $ \mathbb{F}_2[\Sigma_n] $-modules $ \widetilde{\FN}^{*}_{\Sigma_n} $ defined as followed. As a graded $ \mathbb{F}_2 $-vector space, a graded basis is given by all equivalence classes of planar level trees with $ n $ leaves. The degree of a tree is $ \deg(T) = \sum_v y(v)(r_v-1) $, where the sum is over internal vertices of $ T $ and $ r_v $ is the number of outgoing edges of $ v $.

The symmetric group acts on this basis by permuting the labels of the leaves of trees. Moreover, the differential is computed by identifying couples of adjacent edges and performing a vertex shuffle at the quotient vertex (see \cite[Definition 2.5]{Sinha:13} for the precise definition).

There is an alternative description of $ \widetilde{\FN}^*_{\Sigma_n} $. It is the free $ \mathbb{F}_2[\Sigma_n] $-module with basis given by $ (n-1) $-tuples of non-negative integers $ [a_1,\dots,a_{n-1}] $. A tree $ T $ as above determines such an $ (n-1) $-tuple  $ [a_1,\dots, a_{n-1}] $ by ordering its leaves from left to right according to the $ x $ coordinate, considering the unique couple of adjacent vertices $ (e_i,f_i) $ that separates the $ i^{th} $ and the $ (i+1)^{th} $ leaves, and recording the $ y $ coordinate $ a_i $ of their common vertex.
This is a bijection after forgetting the labels of the leaves, which determines a permutation in $ \Sigma_n $.
The differential can be described directly in terms of the corresponding $ (n-1) $-tuple and the combinatorics of the symmetric groups (see \cite{Salvetti:00}).
With this description, the degree of $ [a_1,\dots,a_{n-1}] $ is simply $ \sum_{i=1}^{n-1} a_i $.

Let $ \mathcal{r} $ be the reflection across the $ y $ axis.
There is also a symmetric version of this complex, whose basic defining objects are planar trees $ T $ satisfying the first four conditions above and having the two additional symmetry properties:
\begin{itemize}
\item $ T $ is symmetric with respect to $ \mathcal{r} $
\item $ T $ has $ (2n+1) $ leaves that are labeled bijectively with the set $ \{-n,\dots,-1,0, $ $ 1,\dots,n \} $, and this labeling is antisymmetric with respect to $ \mathcal{r} $
\end{itemize}
We call such trees antisymmetric planar level trees, and we consider the equivalence relation $ \sim $ defined as in the non-symmetric case.

There is a cochain complex of $ \mathbb{F}_2[\Bgroup_n] $-modules $ \widetilde{\FN}^*_{\Bgroup_n} $ defined as followed. As a graded $ \mathbb{F}_2 $-vector space, a graded basis is given by all equivalence classes of antisymmetric planar trees with $ (2n+1) $ leaves. The degree of a tree is $ \deg(T) = \sum_v y(v)(r_v-1) $, where the sum is over internal vertices of $ T $ that belong to the positive half-plane $ \{(x,y): x \geq 0\} $, and $ r_v $ is the number of outgoing edges of $ v $ contained in that half-plane.

$ \Bgroup_n $ acts on this basis by realizing its elements as signed permutations, and permuting the labels of the leaves of trees. The differential is defined in terms of tree quotients and symmetric vertex shuffles (see \cite[Proposition 3.5]{Guerra:21} for the precise definition).

There is an alternative description of $ \widetilde{\FN}^*_{\Bgroup_n} $. It is the free $ \mathbb{F}_2[\Bgroup_n] $-module with basis given by $ n $-tuples of non-negative integers $ [a_0,\dots,a_{n-1}] $. An antisymmetric planar tree $ T $ as above determines such an $ n $-tuple  $ [a_0,\dots, a_{n-1}] $ similarly to $ \widetilde{\FN}^*_{\Sigma_n} $, but considering only couples of adjacent vertices contained in the positive half-plane.
The differential can be described directly in terms of the corresponding $ n $-tuple and the combinatorics of $ \Bgroup_n $ (see \cite{Salvetti:00}).
With this description, the degree of $ [a_0,\dots,a_{n-1}] $ is simply $ \sum_{i=0}^{n-1} a_i $.

The complex $ \widetilde{\FN}^*_{\Sigma_n} $ (respectively $ \widetilde{\FN}^*_{\Bgroup_n} $) is known to be dual to a free $ \Sigma_n $- (respectively $ \Bgroup_n $-) resolution of $ \mathbb{F}_2 $. Consequently, for every subgroup $ G \leq \Sigma_n $ (respectively $ G \leq \Bgroup_n $), the cohomology of the quotient $ \widetilde{\FN}^*_{\Sigma_n}/G $ (respectively $ \widetilde{\FN}^*_{\Bgroup_n}/G $) is isomorphic to $ H^*(G; \mathbb{F}_2) $.
These complexes have a special name.
\begin{definition} \label{def:FN}
Let $ G \leq \Sigma_n $ (respectively $ G \leq \Bgroup_n $) be a subgroup. The Fox-Neuwirth cochain complex of $ G $, denoted $ \FN^*_G $, is the quotient $ \widetilde{\FN}^*_{\Sigma_n}/G $ (respectively $ \widetilde{\FN}^*_{\Bgroup_n}/G $).
\end{definition}

\subsection{Linking \texorpdfstring{$ \FN^*_{\mathcal{A}_n} $}{the Fox-Neuwirth resolutions for the alternating subgroups of the symmetric groups} and \texorpdfstring{$ \FN^*_{\Bgrouppos{n}} $}{Bn}}

In this subsection, we analyze the relation between Fox-Neuwirth complexes of the classical alternating group $ \mathcal{A}_n \leq \Sigma_n $ and $ \Bgrouppos{n} $. We preliminarily provide a more explicit description of $ \FN^*_{\Bgrouppos{n}} $ and $ \FN^*_{\mathcal{A}_n} $.
\begin{definition} \label{def:FN+}
Fix a reflection $ s \in \Bgroup_n $. Given an $ n $-tuple $ [a_0,\dots,a_{n-1}] $, we define $ [a_0,\dots,a_{n-1}]^+ $ and $ [a_0,\dots, a_{n-1}]^- $ as the elements in $ \FN_{\Bgrouppos{n}}^* $ represented by $ [a_0,\dots,a_{n-1}] \in \widetilde{\FN}^*_{\Bgroup_{n}} $ and $ s.[a_0,\dots,a_{n-1}] \in \widetilde{\FN}^*_{\Bgroup_{n}} $, respectively.
With tree notation, if $ T $ is a planar level tree with $ 2n+1 $ leaves and symmetric with respect to $ \mathcal{r} $, then we define $ T^+ $ as the element in $ \FN_{\Bgrouppos{n}}^* $ represented by $ T $ with the antisymmetric labeling of its leaves given by placing the labels $ -n,\dots,n $ from left to right in order, and we define $ T^- = s.T^+ $.
\end{definition}

A direct argument, similar to that of the proof of Proposition \ref{prop:involution}, shows that $ [a_0,\dots, a_{n-1}]^+ $ and $ [a_0,\dots, a_{n-1}]^- $ do not depend on the choice of the reflection $ s $. Moreover, since $ \{1,s\} $ is a set of representatives for the cosets $ \Bgroup_n/\Bgrouppos{n} $, these signed $ n $-tuples $ [a_0,\dots,a_{n-1}]^\pm $ form a graded basis of $ \FN_{\Bgrouppos{n}}^* $ over $ \mathbb{F}_2 $.
The differential is computed via symmetric vertex shuffles, by keeping track of the signs of the elements of $ \Bgroup_n $ involved.

A similar description in terms of signed $ (n-1) $-tuples exists for $ \FN^*_{\mathcal{A}_n} $ (see \cite[Section 4]{Sinha:17}).

To link the Fox-Neuwirth complexes for $ \mathcal{A}_n $ and $ \Bgrouppos{n} $, we need to recall that there is a $ k $-symmetrization map $ \tilde{S}_k \colon \widetilde{\FN}^*_{\Sigma_n} \to \widetilde{\FN}^*_{\Bgroup_n} $ that attach a planar level tree $ T $ to the $ y $ axis from the right up to height $ k $ and then symmetrizes (see \cite[Definition 3.2]{Guerra:21}).
Although $ \tilde{S}_k $ does not commute with the differential, the following statement is proved from the definition by direct inspection.

\begin{proposition} \label{prop:symmetrization}
Let $ d_0 \colon \FN^*_{\Bgrouppos{n}} \to \FN^{*+1}_{\Bgrouppos{n}} $ be the map defined as follows. If $ T $ is an antisymmetric planar level tree, let $ (e,f) $ be the couple of adjacent edges in $ T $ separating the central leaf with the leaf immediately to its right. $ d_0(T) = \sum_{\sigma} \sigma T/(\mathcal{r}(f)=e=f) $, where the sum is over all symmetric vertex shuffles at the quotient vertex.

Let $ \mathcal{A}_n \leq \Sigma_n $ be the alternating subgroup.
$ \tilde{S}_k $ induces a morphism of graded $ \mathbb{F}_2 $-vector spaces $ S_k \colon \FN^*_{\mathcal{A}_n} \to \FN_{\Bgrouppos{n}}^* $.
Moreover, $ S_0 $ satisfies the following identity:
\[
\forall x \in \FN^*_{\mathcal{A}_n}: \quad dS_0(x) = S_0(dx) + d_0(S_0(x))
\]
\end{proposition}

By Proposition \ref{prop:symmetrization}, understanding the differential of a $ 0 $-symmetrization in $ \FN_{\Bgrouppos{n}}^* $ boils down to the computation of $ d_0 $.
In general the combinatorics of $ d_0 $ in $ \FN_{\Bgrouppos{n}}^* $ is complicated.
Luckily, we will only use a special case, that we analyze below.
\begin{lemma} \label{lem:d0}
Let $ r \geq 2 $. If $ \underline{a} = [a_0,\dots,a_{n-1}] $ be an $ n $-tuple such that $ a_0 = 0 $, $ a_1 = \dots = a_{r-1} = 1 $ and $ a_{r} = 0 $, then
$ d_0(\underline{a}^+) = d_0(\underline{a}^-) = 0 $.
\end{lemma}
\begin{proof}
Let $ x = \underline{a}^\pm $, and let $ \underline{b} $ be the $ n $-tuple obtained from $ \underline{a} $ by replacing $ a_0 $ with $ 1 $. Then, by unraveling the definition,
\[
d_0(x) = \sum_{\sigma \in S^+} \underline{b}^\pm + \sum_{\sigma \in S^-} \underline{b}^\mp,
\]
where $ S^+ $ (respectively $ S^- $) is the set of shuffles of the sets $ \{-r,\dots,-1\} $, $ \{0\} $ and $ \{ 1,\dots,r \} $ that are signed permutations of $ \{-r,\dots,r\} $ and have positive (respectively negative) sign.

We consider the function $ f \colon S^+ \to S^- $ given by precomposition with the transposition $ (-1,1) $.
We observe that $ f $ is well-defined and bijective. Therefore $ |S^+| = |S^-| $.
There are exactly $ 2^r $ symmetric shuffles of $ \{-r,\dots,-1\} $, $ \{0\} $ and $ \{ 1,\dots,r \} $, hence $ |S^+| = |S^-| = 2^{r-1} $. If $ r \geq 2 $, this is an even number. Consequently, $ d_0(x) = 0 $.
\end{proof}

\subsection{Construction of almost-Hopf semiring generators}

After having recalled the definition and some properties of Fox-Neuwirth complexes, we construct generators for the almost-Hopf ring $ \widetilde{\AB} $ by defining some representatives in $ \FN^*_{\Bgrouppos{n}} $.

\begin{definition} \label{def:generators}
For all $ l \geq 2 $ and $ m \geq 1 $, let $ \Gamma_{l,m}^+ $ and $ \Gamma_{l,m}^- $ be the cochains in $ \widetilde{\FN}_{\Sigma_{m2^l}}/\mathcal{A}_{m2^l} $ described in Definition 6.1 of \cite{Sinha:17}.
Explicitly, let $ * $ denote concatenation of tuples, and let $ B_{k,i,j} = [1]^{* (2^l-1)} $ if $ k \notin \{i,j\} $, $ B_{i,i,j} = [2] $, and $ B_{j,i,j} = [1]^{* (2^l-3)} $.
Also let $ \Gamma_{l,m}^\pm = \alpha_{l,m}^\pm + \sum_{1 \leq i < j \leq (m+1)} \beta_{l,m}^0(i,j) $, where:
\begin{gather*}
\alpha_{l,m}^{\pm} = \left(\underbrace{[1]^{* (2^l-1)} * [0] * \dots * [1]^{*(2^l-1)}}_{m \mbox{ times}} \right)^\pm \\
\beta_{l,m}^0(i,j) = \left( B_{1,i,j} * [0] * B_{2,i,j} * \dots * B_{m+1,i,j} \right)^+ + \left( B_{1,i,j} * [0] * B_{2,i,j} * \dots * B_{m+1,i,j} \right)^-
\end{gather*}
Define $ \widehat{\Gamma}^\pm_{l,m} = S_0(\Gamma^\pm_{l,m}) \in \FN_{\Bgrouppos{m2^l}}^* $.
\end{definition}

As a consequence of our calculations of $ d_0 $ in $ \FN_{\Bgrouppos{n}}^* $ we obtain that $ \widehat{\Gamma}^\pm_{l,m} $ is a cocycle.
\begin{corollary} \label{cor:cocycle}
For all $ l \geq 2 $ and $ m \geq 1 $, $ d(\widehat{\Gamma}^+_{l,m}) = d(\widehat{\Gamma}^-_{l,m}) = 0 $.
\end{corollary}
\begin{proof}
By Proposition \ref{prop:symmetrization} and \cite[Proposition 6.2]{Sinha:17}, it is enough to prove that $ d_0(d(\widehat{\Gamma}^+_{l,m})) = d_0(d(\widehat{\Gamma}^-_{l,m})) = 0 $.
By Lemma \ref{lem:d0}, $ d_0(S_0(\alpha_{l,m}^\pm)) = 0 $ and, if $ i \not= 1 $, $ d_0(S_0(\beta_{l,m}^0(i,j))) = 0 $. If $ i = 1 $, we still have that $ d_0(S_0(\beta_{l,m}^0(1,j))) = 0 $ by a direct calculation (in such case, $ d_0(S_0(\beta_{l,m}^0(1,j))) $ is a sum of four addends, with two distinct terms each appearing twice).
\end{proof}

Thanks to Corollary \ref{cor:cocycle}, the following definition is well-posed.
\begin{definition} \label{def:gamma charged}
For $ l \geq 2 $ and $ m \geq 1 $, we define $ \gamma_{l,m}^+ $ and $ \gamma_{l,m}^- $ as the cohomology classes in $ H^{m(2^l-1)}(\Bgrouppos{m2^l}; \mathbb{F}_2) = H^{m(2^l-1)}(\FN_{\Bgrouppos{m2^l}}^*) $ represented by the cocycles $ \widehat{\Gamma}_{l,m}^+ $ and $ \widehat{\Gamma}_{l,m}^- $, respectively.
\end{definition}

A priori, the notation of Definition \ref{def:gamma charged} could conflict with that of Theorem \ref{thm:basis alternating subgroup}. The following observation, proved by direct inspection, guarantees that this conflict does not occur.

\begin{corollary} \label{cor:transfer gamma charged}
$ \tr(\gamma_{l,m}^+) = \tr(\gamma_{l,m}^-) = \gamma_{l,m} $ and $ \res(\gamma_{l,m}) = \gamma_{l,m}^+ + \gamma_{l,m}^- $.
\end{corollary}

Another potential notational conflict is with the generators of the cohomology of alternating groups defined in \cite{Sinha:17}. This issue is non-existent thanks to the following.

\begin{corollary} \label{cor:restriction alternating groups}
For all $ k,l \geq 1 $, the restriction of $ \gamma_{k,l}^\pm \in H^*(\Bgrouppos{l2^k}; \mathbb{F}_2) $ to the cohomology ring $ H^*(\mathcal{A}_{l2^k}; \mathbb{F}_2) $ is the class $ \gamma_{k,l}^\pm $ of Giusti--Sinha \cite{Sinha:17}.
\end{corollary}
\begin{proof}
It follows immediately from Definition \ref{def:generators}, Definition 6.1 of \cite{Sinha:17} and Lemma 4.5 of \cite{Guerra:21}.
\end{proof}

\subsection{The transfer product of the \texorpdfstring{classes $ \gamma_{l,m}^\pm $}{generators}}

The chain-level transfer product formula in $ \FN^*_{\Bgroup_n} $ stated in Proposition 4.11 of \cite{Guerra:21} lifts to  $ \FN^*_{\Bgrouppos{n}} $ by keeping track of charges.
This immediately provides the following formulas.
\begin{proposition} \label{prop:transfer gamma+-}
	In $ \widetilde{\AB} $, for all $ k \geq 2 $ and $ l,m \geq 1 $,
	\begin{gather*}
		\gamma_{k,l}^+ \odot \gamma_{k,m}^+ = \gamma_{k,l}^- \odot \gamma_{k,m}^- = \left( \begin{array}{c} l+m \\ m \end{array} \right) \gamma_{k,l+m}^+ \\
		\tag*{and} \gamma_{k,l}^+ \odot \gamma_{k,m}^- = \gamma_{k,l}^- \odot \gamma_{k,m}^+ = \left( \begin{array}{c} l+m \\ m \end{array} \right) \gamma_{k,l+m}^-,
	\end{gather*}
\end{proposition}

\subsection{The coproduct of the classes \texorpdfstring{$ \gamma_{l,m}^\pm $}{generators}}

In this sub-section, we obtain formulas for the coproduct of $ \gamma_{l,m}^\pm $. The proof is sensibly more complicated than the transfer product identities of Proposition \ref{prop:transfer gamma+-}, because the chain-level lift of the coproduct is combinatorially involved.

By \cite[Proposition 4.9]{Guerra:21}, there is a cochain-level lifting of the coproduct map $ \Delta \colon \FN^*_{\Bgroup_{n+m}} \to \FN^*_{\Bgroup_n} \otimes \FN^*_{\Bgroup_m} $ defined in terms of the symmetric pruning maps.
Symmetric $ k $-prunings are defined in \cite[Definition 4.7]{Guerra:21}. Shortly, a symmetric $ k $-pruning of an antisymmetric planar level tree $ T $ is a couple $ (T',T'') $, where $ T' $ is obtained by removing from the internal vertices of $ T $ of height $ k $ some of their leftmost or rightmost child subtrees in a symmetric way, and $ T'' $ is a vertex shuffle of the removed subtrees joined together at a single vertex $ v $ with $ y(v) = k $.

We construct a cochain representative of the coproduct of $ \AB $ by defining a sign-extension of the pruning maps.
\begin{definition}
	Let $ k \in \mathbb{N} $. Let $ (T',T'') $ be a $ k $-pruning of an antisymmetric planar level tree $ T $ with $ 2n + 1 $ leaves. By keeping track of labels, there are induced antisymmetric labelings on the leaves of the trees $ T' $ and $ S_k(T'') $, with labels in sets $ (-I') \cup \{0\} \cup I' $ and $ (-I'') \cup \{0\} \cup I'' $, respectively, such that $ \{1,\dots,n\} = I' \sqcup I'' $.
	
	Let $ (T',T'') $ be a $ k $-pruning of $ T $. We say that a choice of signs $ \varepsilon', \varepsilon'' \in \{-,+\} $ are compatible with the positive (respectively negative) sign on $ T $ if $ T^+ $ (respectively $ T^- $) is represented by an antisymmetric labeling of the leaves of $ T $ with the following properties:
	\begin{itemize}
	\item $ I' = \{1,\dots,r\} $ and $ I'' = \{r+1,\dots,n\} $ for some $ r $,
	\item and after relabeling $ S_k(T'') $ via the unique order-preserving bijection $ I'' \cong \{1,\dots,n-r\} $, the induced labelings on $ T' $ and $ S_k(T'') $ represent $ (T')^{\varepsilon'} $ and $ (T'')^{\varepsilon''} $.
	\end{itemize}
	
	The signed $ k $-pruning map is the linear map $ P^s_k \colon \FN^*_{\Bgrouppos{n}} \to \bigoplus_{r=0}^n \FN^*_{\Bgrouppos{r}} \otimes \FN^*_{\Bgrouppos{n-r}} $ defined by
	\[
	P^s_k(T^\pm) = \sum_{(T',T''),(\varepsilon',\varepsilon'')} (T')^{\varepsilon'} \otimes (S_k(T''))^{\varepsilon''},
	\]
	where the sum runs over $ k $-prunings $ (T',T'') $ and choice of signs $ \varepsilon',\varepsilon'' $ compatible with the sign of $ T^\pm $.
	
	The signed concatenation map $ C^s \colon \FN^*_{\Bgrouppos{n}} \otimes \FN^*_{\Bgrouppos{m}} \to \FN^*_{\Bgrouppos{n+m}} $ is defined as the oriented concatenation map of \cite[Definition 4.7]{Guerra:21}.
\end{definition}
We observe that every non-trivial pruning of $ T^\pm $ admits exactly two choices of signs compatible with the sign of $ T^{\pm} $.

The following proposition is proved with a variation of the argument used for \cite[Proposition 4.14]{Guerra:21}.
\begin{proposition} \label{prop:coproduct FN}
	Define $ \Delta_k \colon \FN^*_{\Bgrouppos{n}} \to \bigoplus_{r=0}^n \FN^*_{\Bgrouppos{r}} \otimes \FN^*_{\Bgrouppos{n-r}} $ by the recursive formulas:
	\begin{itemize}
		\item $ \Delta_0 = P_0^s $
		\item $ \Delta_k = (\id \otimes C^s) (P^s_k \otimes \id) \Delta_{k-1} $ if $ k \geq 1 $
	\end{itemize}
	$ \Delta_k $ stabilizes and the limit map $ \Delta $ is a cochain map inducing the coproduct of $ \AB $ in cohomology.
\end{proposition}

In order to compute the coproduct of $ \gamma_{l,m}^\pm $, we require the following elementary computational lemma.
\begin{definition}[from \cite{Sinha:13}]
	Let $ \underline{a} = [a_1,\dots,a_{n-1}] $ be an $ (n-1) $-tuple of non-negative integers. A $ k $-block in $ \underline{a} $ is a maximal interval $ [a_r,\dots,a_s] $ such that $ a_i > k $ for all $ r \leq i \leq s $.
\end{definition}

\begin{lemma} \label{lem:higher prunings}
	Let $ \underline{a} = [a_1,\dots,a_{n-1}] $ be an $ (n-1) $-tuple of non-negative integers. Let $ k \in \mathbb{N} $. Assume that the $ k $-blocks of $ \underline{a} $ are either sequences of repeated $ (k+1) $, or intervals without entries equal to $ (k+1) $.
	Then $ P^s_{k+1}(S_0(\underline{a}^+)) = P^s_{k+1}(S_0(\underline{a}^-)) = 0 $.
\end{lemma}
\begin{proof}
	$ k $-blocks correspond to child subtrees at vertices of height $ k $ in the positive half-plane. We let $ T_1, \dots, T_l $ be those corresponding to $ k $-blocks of the form $ [(k+1),\dots,(k+1)] $. Thus, for all $ 1 \leq m \leq l $, $ T_m $ is a subtree with a number $ r_m $ of leaves and a single internal vertex of height $ k+1 $. A $ (k+1) $- pruning of $ T $ is a couple $ (T'_{(i_1,\dots,i_l,j_1,\dots,j_l)},T'') $, where $ T'_{(\underline{i},\underline{j})} $ is obtained by removing $ i_m $ leaves from the left side of $ T_m $, $ j_m $ leaves from its right side and symmetrizing, while $ T'' $ is obtained by shuffling the removed leaves. If we forget leaves labels, $ T'' $ corresponds to the sequence $ [k+1]^{* \sum_{m=1}^l (i_m+j_m)-1} $.
	We now keep track of sign compatibility. We let
	\begin{align*}
		X &= \{(i_1,\dots,i_l,j_1,\dots,j_l,\sigma): i_m \geq 0, j_m \geq 0, i_m+j_m < r_m, \\
		&\quad\quad\quad\quad\quad\quad\quad\sigma \in \Sh(i_m,\dots,i_1,j_1,\dots,j_m) \}, \\
		\varepsilon &\colon X \to \{-,+\}, \  \varepsilon(\underline{i},\underline{j}, \sigma) = (-)^{\sum_{1 \leq a < b \leq l} (i_a j_b + j_aj_b) + \sum_{m=1}^l (i_m^2 + i_mj_m + \frac{i_m^2+i_m}{2})} \sgn(\sigma)
	\end{align*}
	and we explicitly have that
	\begin{align*}
		P^s_{k+1}(S_0(\underline{a})^\pm) &= \sum_{(\underline{i},\underline{j},\sigma)\in X} \Bigg( {T'_{(\underline{i},\underline{j})}}^\pm \otimes S_{k+1}([k+1]^{* \sum_{m=1}^l (i_m+j_m) - 1})^{\varepsilon(\underline{i},\underline{j},\sigma)} \\
		&+ {T'_{(\underline{i},\underline{j})}}^\mp \otimes S_{k+1}([k+1]^{* \sum_{m=1}^l (i_m+j_m) - 1})^{(-) \cdot \varepsilon(\underline{i},\underline{j},\sigma)} \Bigg).
	\end{align*}
	
	We let $ s_m = i_m+j_m  $ for all $ 1 \leq m \leq l $.
	We consider $ 2l $ adjacent intervals $ I_1,\dots,I_{2l} $ in $ \{1,\dots,\sum_{m=1}^l (i_l+j_k)\} $ such that $ |I_m| = j_{l+1-m} $ if $ 1 \leq m \leq l $ and $ |I_m| = i_{m-l} $ if $ l+1 \leq m \leq 2l $. Hence, $ I_1 = \{1,\dots,j_m\} $, $ I_2 = \{j_m+1,\dots,j_m+j_{m-1}\} $, etcetera.
	Let $ \overline{\sigma}_{(\underline{i},\underline{j})} \in \Sigma_{\sum_{m=1}^l (i_m+j_m)} $ be the block permutation that reverses the order of the blocks $ I_1,\dots, I_{2l} $.
	A direct calculation shows that
	\[
	\sgn(\overline{\sigma}_{(\underline{i},\underline{j})}) = (-)^{\sum_{1 \leq a < b \leq l} s_as_b + \sum_{m=1}^l (i_mj_m)}.
	\]
	
	Let $ \Phi \colon X \to X $ given by $ \Phi(\underline{i}, \underline{j},\sigma) \mapsto (\underline{j},\underline{i},\sigma\overline{\sigma}_{(\underline{i},\underline{j})}) $. We have that
	\begin{align*}
		\frac{\varepsilon(\Phi(\underline{i},\underline{j},\sigma))}{\varepsilon(\underline{i},\underline{j},\sigma)} &= \frac{(-)^{\sum_{1 \leq a \leq b \leq l}(i_aj_b+j_aj_b) + \sum_{m=1}^l (i_m^2+i_mj_m + \frac{i_m^2+i_m}{2})}}{(-)^{\sum_{1 \leq a \leq b \leq l}(j_ai_b+i_ai_b) + \sum_{m=1}^l (j_m^2+j_mi_m + \frac{j_m^2+j_m}{2})}} \sgn(\overline{\sigma}_{(\underline{i},\underline{j})}) \\
		& = (-)^{\sum_{m=1}^l \left( s_m^2 + \frac{s_m^2 + s_m}{2} \right)}.
	\end{align*}
	In particular, this ratio only depends on $ s_1,\dots,s_m $.
	
	To conclude the proof, we consider, for every $ l $-tuple $ (s_1,\dots,s_l) $, the subset $ X_{(s_1,\dots,s_l)} \subseteq X $ containing those $ (\underline{i}',\underline{j}',\sigma') $ satisfying $ i'_m + j'_m = s_m $ for all $ m $. $ \Phi $ restricts to an involution of $ X_{(s_1,\dots,s_l)} $.
	
	We prove that all the terms in $ P^s_{k+1} $ indexed by elements of $ X_{(s_1,\dots,s_l)} $ cancel by distinguishing two cases:
	\begin{itemize}
		\item If $ \sum_{m=1}^l \left( s_m^2 + \frac{s_m^2 + s_m}{2} \right) $ is even, then $ \varepsilon \Phi = \varepsilon $ on $ X_{(s_1,\dots,s_l)} $.
		Since $ \Phi $ has no fixed points in $ X $, in the expression of $ P^s_{k+1}(S_0(\underline{a})^{\pm}) $, the addend corresponding to $ (\underline{i},\underline{j},\sigma) $ cancels itself out with the term corresponding to $ \Phi(\underline{i},\underline{j},\sigma) $.
		\item If $ (\underline{i},\underline{j},\sigma) \in X $ is such that $ \sum_{m=1}^l \left( s_m^2 + \frac{s_m^2 + s_m}{2} \right) $ is odd, then the addends of $ P^s_{k+1} $ corresponding to $ X_{(s_1,\dots,s_l)} $ are all equal up to their sign, and $ \Phi $ exchanges such signs. Therefore, each of such terms appears $ |X_{(s_1,\dots,s_l)}|/2 $ times.
		Since, in this case, $ |X_{(s_1,\dots,s_l)}| = \prod_{m=1}^l 2^{s_m} $ is always a multiple of $ 4 $, all these terms cancel each other out.
	\end{itemize}
\end{proof}

\begin{proposition} \label{prop:coproduct gamma+-}
	For all $ l \geq 2 $ and $ m \geq 1 $, the following formula holds in $ \widetilde{\AB} $, with the convention that $ \gamma_{l,0}^\pm = 1^{\pm} $:
	\[
	\Delta(\gamma_{l,m}^\pm) = \sum_{i=0}^m \left( \gamma_{l,i}^+ \otimes \gamma_{l,m-i}^\pm + \gamma_{l,i}^- \otimes \gamma_{l,m-i}^\mp \right)
	\]
\end{proposition}
\begin{proof}
	As a notational convention, for an $ n $-tuple $ \underline{a} = [a_0,\dots,a_{n-1}], $ we let $ \underline{a}^0 = \underline{a}^+ + \underline{a}^- $.
	In addition to the cochains $ \alpha_{l,m}^\pm $ and $ \beta_{l,m}^0(i,j) $, we also consider the following cochains in $ \bigoplus_n \FN^*_{\Sigma_n} $:
	\begin{gather*}
		\delta_k(i) = \underbrace{[1]^{*(2^l-1)} * [0] * \dots * [1]^{* (2^l-1)}}_{(i-1) \mbox{ blocks}} * [0] * [2] * [0] * \underbrace{[1]^{* (2^l-1)} * \dots * [0] * [1]^{* (2^l-1)}}_{(k-i) \mbox{ blocks}}\\
		\delta'_k(i) = \underbrace{[1]^{*(2^l-1)} * \dots * [0] * [1]^{* (2^l-1)}}_{(i-1) \mbox{ blocks}} * [0] * [1]^{*(2^l-3)} * [0] * \underbrace{[1]^{* (2^l-1)} * \dots * [1]^{* (2^l-1)}}_{(k-i) \mbox{ blocks}}
	\end{gather*}
	By Proposition \ref{prop:coproduct FN} and Lemma \ref{lem:higher prunings}, $ \Delta(\gamma_{l,m}^\pm) $ is represented by
	\begin{gather*}
		P_0^s(S_0(\alpha_{l,m})^\pm) + \sum_{1 \leq i < j \leq m+1} P_0^s(S_0(\beta^0_{l,m}(i,j))) = \sum_{i=0}^m \left( S_0(\alpha_{l,i}^+) \otimes S_0(\alpha_{l,m-i}^\pm) \right. \\
		\left. + S_0(\alpha_{l,i}^-) \otimes S_0(\alpha_{l,m-i}^{\mp}) \right) + \sum_{1 \leq i < j \leq m+1} \left( \sum_{k=0}^{i-1} (S_0(\alpha_{l,k}^0) \otimes S_0(\beta_{l,m-k}^0(i-k,j-k))) \right. \\
		\left. + \sum_{k=j}^{m+1} ( S_0(\beta_{l,k-1}^0(i,j)) \otimes S_0(\alpha^0_{l,m+1-k})) + \sum_{k=i}^{j-1} ( S_0(\delta_k(i)^0) \otimes S_0({\delta'_{m+1-k}(j-k)}^0)) \right) \\
		= \sum_{i=0}^m \left( \widehat{\Gamma}_{l,i}^+ \otimes \widehat{\Gamma}_{l,m-i}^\pm + \widehat{\Gamma}_{l,i}^- \otimes \widehat{\Gamma}_{l,m-i}^\mp \right) \\
		+ \sum_{k=1}^m \left[ \left( \sum_{i=1}^k S_0(\delta_k(i))^0 \right) \otimes \left( \sum_{h=1}^{m+1-k} S_0({\delta'_{m+1-k}(h)})^0 \right) \right].
	\end{gather*}
	
	We note that the linear map $ \underline{a} \in \FN^*_{\Bgroup_n} \mapsto \underline{a}^0 \in \FN^*_{\Bgrouppos{n}} $ is a cochain lifting of the restriction $ \res_n \colon H^*(\Bgroup_n; \mathbb{F}_2) \to H^*(\Bgrouppos{n}; \mathbb{F}_2) $.
	Moreover, for every $ 1 \leq k \leq m $, $ \sum_{h=1}^{m+1-k} S_0(\delta'_{m+1-k}(h))^0 $ is a cocycle in $ \FN^*_{\Bgroup_{(m+1-k)2^l - 2}} $ and $ \sum_{i=1}^k S_0(\delta_k(i)) $ represents the cohomology class $ \gamma_{l,k-1} \odot \gamma_{1,1}^2 \in H^*(\Bgroup_{(k-1)2^l+2}; \mathbb{F}_2) $. This class belongs to the ideal  generated by $ \gamma_{1,1} \odot 1_{(k-1)2^l} + w \odot 1_{(k-1)2^l+1} $, the Euler class of the double covering considered in the proof of Theorem \ref{thm:basis alternating subgroup}, which is equal to $ \ker(\res_{(k-1)2^l+2}) $.
	We deduce that $ \sum_{i=1}^k S_0(\delta_k(i))^0 $ is a coboundary.
	Consequently
	\[
	\sum_{k=1}^m \left[ \left( \sum_{i=1}^k S_0(\delta_k(i))^0 \right) \otimes \left( \sum_{h=1}^{m+1-k} S_0({\delta'_{m+1-k}(h)})^0 \right) \right]
	\]
	is a coboundary and the desired formula is proved by passing to cohomology.
\end{proof}

\subsection{An explicit charged Hopf monomial basis\texorpdfstring{ for $ \widetilde{\AB} $}{}} \label{subsec:basis}

Up to now, the basis $ \mathcal{M}_{charged} $ of Theorem \ref{thm:basis alternating subgroup} depends on choices of classes $ x^+ $ and $ x^- $ for $ x \in \mathcal{G}_{ann} $. We use what we proved so far on the elements $ \gamma_{l,m}^\pm $ to fix a standard choice for such classes. This way, we construct $ \mathcal{M}_{charged} $ explicitly. In the remainder of this article, we will always use the charged Hopf monomials constructed in this section without further notice.

\begin{definition} \label{def:explicit charged basis}
	Let $ b = (w\dip{m2^n})^{a_0} \prod_{i=1}^n (\gamma_{i,m2^{n-i}})^{a_i} $ be a decorated gathered block with $ n \geq 2 $ and $ a_n \not= 0 $.
	We define
	\[
	b^+ = \res_{m2^n}((w\dip{m2^n})^{a_0} \gamma_{1,m2^{n-1}}^{a_1}) \sum_{\substack{0 \leq k_0 \leq a_0 \\
			\dots \\
			0 \leq k_{n-1} \leq a_{n-1} \\
			0 \leq k_n \leq \lfloor \frac{a_n}{2} \rfloor}}^{a_{n-1}} \prod_{i=2}^n \left( \begin{array}{c}
		a_i \\
		k_i
	\end{array} \right) (\gamma_{i,m2^{n-i}}^-)^{k_i} (\gamma_{i,m2^{n-i}}^+)^{a_i-k_i}.
	\]
	We also define $ b^- = \iota(b^+) $.
	
	If $ x = b_1 \odot \dots \odot b_r \in \mathcal{G}_{ann} $ is a decorated Hopf monomial with constituent gathered blocks $ b_i $, we define
	\[
	x^+ = b_1^+ \odot \dots \odot b_r^+.
	\]
	We also define $ x^- = \iota(x^+) $.
\end{definition}

To prove that these choice of charged Hopf monomials provide a basis, we need to show that the conditions of Theorem \ref{thm:basis alternating subgroup} are satisfied.
\begin{corollary} \label{cor:no conflict}
	For all $ x \in \mathcal{G}_{ann} $, $ \tr(x^+) = \tr(x^-) = x $ and $ \res(x) = x^+ + x^- $.
\end{corollary}
\begin{proof}
	By Proposition \ref{prop:identities restriction transfer} (6) we reduce to the special case where $ a_0 = a_1 = 0 $. The statement then follows by applying inductively Proposition \ref{prop:relations involution} (4), Proposition \ref{prop:identities restriction transfer} (5) and Corollary \ref{cor:transfer gamma charged}.
\end{proof}

\section{The algebraic and invariant-theoretic picture: detection by elementary abelian subgroups}

If $ A \leq \Bgrouppos{n} $ is a subgroup, there is a restriction map in cohomology $ H^*(\Bgrouppos{n}; \mathbb{F}_2) \to H^*(A; \mathbb{F}_2) $.
The final goal of this section is to prove the following detection result for the cohomology of $ \Bgrouppos{n} $.

\begin{theorem}[Detection theorem] \label{thm:detection}
The direct sum of the restriction maps
\[
\rho \colon H^*(\Bgrouppos{n}; \mathbb{F}_2) \to \bigoplus_A H^*(A; \mathbb{F}_2),
\]
where the direct sum runs over all elementary abelian $ 2 $-subgroups $ A \leq \Bgrouppos{n} $, is injective.
\end{theorem}

We recall that an elementary abelian $ 2 $-group is an abelian group where every non-trivial element has order $ 2 $.

\subsection{Recollections on the classification of maximal elementary abelian \texorpdfstring{$ 2 $}{2}-subgroups in \texorpdfstring{$ \Bgroup_n $}{Bn}}

We first recall a result on maximal elementary abelian $ 2 $-subgroups in $ \Bgroup_n $ from Swenson's thesis \cite{Swenson}. In this regard, we depart from it only by adopting the slightly different notation from Guerra \cite{Guerra:21}.

Let $ V_k = \mathbb{F}_2^k $. The regular action of $ V_k $ on itself yields a monomorphism $ V_k \hookrightarrow \Sigma_{2^k} $, defined up to conjugacy.
\begin{definition}[from \cite{Swenson,Guerra:21}]
We say that a partition $ \pi $ of $ n $ is admissible if its parts are all powers of $ 2 $.
Given an admissible partition $ \pi $ of $ n $, we let $ m_k(\pi) $, or simply $ m_k $ where there is no risk of confusion, be the multiplicity of $ 2^k $ in $ \pi $.
\end{definition}

For all $ k \geq 0 $, there is an elementary abelian $ 2 $-subgroup $ A_{(2^k)} = C_k \times V_k \leq \Bgroup_{2^k} $, where $ C_k \cong \mathbb{F}_2 $ is the center of $ \Bgroup_{2^k} $.
If $ \pi = \{ 2^{k_1}, \dots, 2^{k_i} \} $ is an admissible partition of $ n $, then we consider the subgroup
\[
A_{\pi} = \prod_{j=1}^i A_{(2^{k_j})} \leq \prod_{j=1}^i \Bgroup_{2^{k_j}} \leq \Bgroup_n.
\]

The subgroup $ A_{\pi} $ depends on an ordering of the parts of $ \pi $, but its conjugacy class doesn't.

\begin{theorem}[{\cite[Theorem 5.3.4]{Swenson}}] \label{thm:elementary abelian subgroups B}
	Every maximal elementary abelian $ 2 $-subgroup of $ \Bgroup_n $ is conjugate to $ A_{\pi} $ for one and only one admissible partition.
	Moreover, the invariant subalgebra with respect to the conjugacy action of the normalizer is given by
	\[
	\left[ H^*(A_{\pi}; \mathbb{F}_2) \right]^{N_{\Bgroup_n} (A_{\pi})} \cong \bigotimes_{k \in \mathbb{N}} \left[ \mathbb{F}_2[f_{2^k},d_{2^k-1},\dots,d_{2^k-2^{k-1}}]^{\otimes m_k} \right]^{\Sigma_{m_k}},
	\]
	where generators are indexed by their degree and $ m_k $ is the multiplicity of $ 2^k $ in $ \pi $.
\end{theorem}

\begin{remark} \label{rem:invariants}
The classes $ f_{2^k} $ and $ d_{2^k-2^i} $ in the statement of Theorem \ref{thm:elementary abelian subgroups B} are described explicitly as follows.
It is known that the cohomology of an abelian $ 2 $-group $ A $ is isomorphic to the symmetric algebra freely generated by the dual $ \mathbb{F}_2 $-vector space $ A^\vee $ in degree $ 1 $ (see, for instance, \cite[Theorem II.4.4]{Adem-Milgram}).
Let $ \{y_1,\dots,y_k\} $ be a basis of the dual $ \mathbb{F}_2 $-vector space $ V_k^\vee $, and let $ x \in A_{(2^k)}^\vee $ be a linear functional restricting to zero on $ V_k $ and non-zero on $ C_k $. Then
\[
H^*(A_{(2^k)}; \mathbb{F}_2) \cong \mathbb{F}_2[x,y_1,\dots,y_k].
\]
With this notation, $ d_{2^k-1}, \dots, d_{2^k-2^{k-1}} $ are the invariants introduced by Dickson (see \cite{Wilkerson} for a nice treatment) and
\[
f_{2^k} = \prod_{v \in V_k^\vee} (x-v).
\]
For example, in the cohomology of $ A_{(4)} $, $ f_4 = x(x-y_1)(x-y_2)(x-y_1-y_2) $.
\end{remark}

\subsection{Classification of maximal elementary abelian \texorpdfstring{$ 2 $}{2}-subgroups in \texorpdfstring{$ \Bgrouppos{n} $}{the alternating subgroups of Bn}}

We now study the elementary abelian $ 2 $-subgroups of $ \Bgrouppos{n} $ and we provide a complete characterization up to conjugacy. We need a preliminary definition.

\begin{definition}
We say that an admissible partition $ \pi $ is irrelevant if the multiplicity of $ 1 $ in $ \pi $ is different from $ 2 $ or the multiplicity of $ 2 $ in $ \pi $ is different from $ 0 $.
We say that an admissible partition is relevant if it is not irrelevant.
\end{definition}

We build our analysis of $ \Bgrouppos{n} $ on Theorem \ref{thm:elementary abelian subgroups B}. To do this, we first record the following well-known elementary fact about groups.
\begin{proposition} \label{prop:conjugacy classes}
Let $ G $ be a group and $ H \leq G $ be a subgroup. There is a bijection between $ G/N_G(H) $ and the conjugacy class of $ H $, given by $ g \cdot N_G(H) \mapsto gHg^{-1} $.
\end{proposition}

\begin{definition} \label{def:elementary B+}
For an admissible partition $ \pi $ of $ n $, we define $ \widehat{A}_{\pi} = A_{\pi} \cap \Bgrouppos{n} $.
\end{definition}

The subgroups $ \widehat{A}_{\pi} $ admit the following simple description.
\begin{proposition} \label{prop:description A^}
Let $ \pi $ be an admissible partition of $ n $. Let $ p \colon A_{(2)} = C_1 \times V_1 \to V_1 $ be the projection onto the second factor.
Let $ \phi_{m_0,m_1} $ be the homomorphism:
\[
\phi_{m_0,m_1} \colon A_{(1)}^{m_0} \times A_{(2)}^{m_1} = C_0^{m_0} \times (C_1 \times V_1)^{m_1} \stackrel{\id_{C_0}^{m_0} \times p^{m_1}}{\xrightarrow{\hspace{4em}}} C_0^{m_0} \times V_1^{m_1} \cong \mathbb{F}_2^{m_0+m_1} \stackrel{\Sigma}{\rightarrow} \mathbb{F}_2
\]
Then, as a subgroup of $ A_{\pi} $, $ \widehat{A}_{\pi} = \ker(\phi_{m_0,m_1}) \times A_{\pi'} $, where $ \pi' $ is obtained from $ \pi $ by removing all occurrences of $ 1 $ and $ 2 $.
\end{proposition}
\begin{proof}
We first assume that $ \pi = {2^k} $ contains a single element. Then $ n = 2^k $ and $ A_{\pi} = C_k \times V_k $.
$ C_k $ is generated by the product of the $ n $ reflections across the coordinate hyperplanes, which has positive sign if and only if $ k > 0 $.
$ V_k $ is generated by translations in $ \mathbb{F}_2^k $ by coordinate vectors, which are the product of $ 2^{k-1} $ disjoint inversions. Therefore $ V_k \subseteq \Bgrouppos{n} $ if $ k \geq 2 $, and is generated by a single element of negative sign if $ k = 1 $.
In particular, $ A_{(2^k)} \subseteq \Bgrouppos{2^k} $ if and only if $ k \geq 2 $.
		
Consequently, for a general $ \pi $, if $ m_0 = m_1 = 0 $, then $ A_{\pi} = \widehat{A}_{\pi} $. If $ m_0 $ or $ m_1 $ are non-zero, then
\[
A_{\pi} = A_{(1)}^{m_0} \times A_{(2)}^{m_1} \times A_{\pi'} \subseteq \Bgroup_{m_0+2m_1} \times \Bgrouppos{n-m_0-2m_1}.
\]
Hence $ \widehat{A}_{\pi} \subseteq (\Bgroup_{m_0+2m_1} \times \Bgrouppos{n-m_0-2m_1}) \cap \Bgrouppos{n} = \Bgrouppos{m_0+2m_1} \times \Bgrouppos{n-m_0-2m_1} $.
We deduce that
\[
\widehat{A}_{\pi} = \left[(A_{(1)}^{m_0} \times A_{(1)}^{m_1}) \cap \Bgrouppos{m_0+2m_1} \right] \times A_{\pi'} = \ker(\sgn|_{A_{(1)}^{m_0} \times A_{(2)}^{m_1}}) \times A_{\pi'}.
\]
The statement follows by the immediate identification between $ \phi_{m_0,m_1} $ and the restriction of the sign homomorphism.
\end{proof}

We will also need the following result, regarding the action of the subgroups $ A_{\pi} $ on $ \{1,\dots,n\} $ induced by the projection $ \Bgroup_n \to \Sigma_n $.
\begin{lemma} \label{lem:orbits conjugates}
Let $ \pi $ and $ \mu $ be relevant admissible partitions of $ n $. Let $ x \in \Bgroup_n $. If $ x \widehat{A}_{\pi} x^{-1} \subseteq A_{\mu} $, then $ x $ induces a bijection between the orbit sets
\[
\{ 1,\dots, n \}/A_{\pi} \cong \{1,\dots, n\}/A_{\mu}.
\]
\end{lemma}
\begin{proof}
To fix notation, we let $ K $ be the kernel of the projection $ \Bgroup_n \to \Sigma_n $.

The explicit description of Proposition \ref{prop:description A^} implies that, unless $ m_0 = 0 $ and $ m_1 = 1 $, the actions of $ \widehat{A}_{\pi} $ and $ A_{\pi} $ on $ \{1,\dots,n\} $ have the same orbits. Therefore, in such cases $ x(O) $ is contained in some $ O' \in \{1,\dots,n\}/A_{\mu} $ for all $ O \in \{1,\dots,n\}/A_{\pi} $, and it is enough to prove that $ x(O) = O' $.

If $ m_0 = 0 $ and $ m_1 = 1 $, then the action of $ A_{\pi} $ has an orbit of order $ 2 $, but not $ \widehat{A}_{\pi} $. Extra care is needed in this exceptional situation.

We argue case by case, depending on $ m_0 $, $ m_1 $ and  the cardinality of $ O $:
\begin{itemize}
\item If $ |O| > 2 $, then $ O $ is also an orbit for $ \widehat{A}_{\pi} $, and $ |O| = 2^k $ for some $ k \geq 2 $ because all orbits have cardinality equal to a power of $ 2 $. Let $ \sigma \in \Sigma_n $ satisfy $ \sigma(O) = \{1,\dots,2^k\} $. Then, by Proposition \ref{prop:description A^}, $ \sigma \widehat{A}_{\pi} \sigma^{-1} $ contains $ G \times 1_{n-2^k} $, where $ G $ is conjugate to $ A_{(2^k)} $. We can assume, without loss of generality, that $ G = A_{(2^k)} $. Similarly, we can find $ \sigma' \in \Sigma_n $ such that $ \sigma'(O') = \{1,\dots,2^{k'}\} $ and $ A_{(2^{k'})} \times 1_{n-2^{k'}} \subseteq \sigma' A_{\mu} {\sigma'}^{-1} $, with $ |O'| = 2^{k'} $. Then
\[
(A_{(2^k)} \times 1_{n-2^k})^{\sigma x^{-1} {\sigma'}^{-1}} \subseteq (A_{\mu}^{{\sigma'}^{-1}}) \cap (\Bgroup_{2^{k'}} \times 1_{n-2^{k'}}) = A_{(2^{k'})} \times 1_{n-2^{k'}}.
\]

We assume by contradiction that $ O' \not= x(O) $, so there exists $ l \in O' \setminus x(O) $. Let $ y \in A_{(2^k)} $. Then $ y' = (y \times 1_{n-2^k})^{\sigma x^{-1} {\sigma'}^{-1}} $ fixes $ \sigma'(l) $. Since the action of $ V_{2^{k'}} \subseteq A_{(2^{k'})} $ is free on $ \{1,\dots,2^{k'}\} $ and $ C_{k'} \subseteq A_{(2^{k'})} $ acts trivially, $ y' \in C_{k'} \times 1_{n-2^{k'}} \subseteq K $. $ K $ is a normal subgroup, so this would imply that $ A_{(2^k)} \times 1_{n-2^k} \subseteq K $, which contradicts the definition of $ A_{(2^k)} $.
\item If $ m_0 = 0 $, $ m_1 > 2 $ and $ |O| \leq 2 $, then $ \widehat{A}_{\pi} $ has no orbit of cardinality $ 1 $ and $ |O| = 2 $.

Let $ O' \in \{1,\dots,n\}/A_{\mu} $ be the orbit containing $ x(O) $.  $ O' = \bigsqcup_{O''} x(O'') $ for some family of orbits $ O'' $ of $ \widehat{A}_{\pi} $.
Since all orbits have cardinality a power of $ 2 $, if $ O' $ were bigger that $ x(O) $, there would be another orbit $ O'' \not= O $ such that $ x(O'') \subseteq O' $ and $ |O''| = 2 $.
Under such condition we could find, with the same argument used in the previous case, permutations $ \sigma, \sigma' \in \Sigma_n $ such that
\[
(\widehat{A}_{(2,2)} \times 1_{n-4})^{\sigma x^{-1} \sigma'^{-1}} \subseteq A_{(|O'|)} \times 1_{n-|O'|}.
\]
Since $ K $ is a normal subgroup, this would imply that
\[
\left[ K \cap (\widehat{A}_{(2,2)} \times 1_{n-4}) \right]^{\sigma x^{-1} {\sigma'}^{-1}} \subseteq K \cap (A_{|O'|} \times 1_{n-|O'|}),
\]
by this would be contradictory because the former has $ 4 $ elements and the latter $ 2 $. We must conclude that $ x(O) = O' $ in this case.
\item If $ m_0 = 0 $, $ m_1 \leq 1 $ and $ |O| \leq 2 $, then $ m_1 = 1 $ and $ O $ is the unique orbit of cardinality $ 2 $. Even if $ \{1,\dots,n\}/\widehat{A}_{\pi} \not= \{1,\dots,n\}/A_{\pi} $ in this case, by what we proved so far, $ x $ induces a bijection on all other orbits. This implies $ x(O) = O' $.
\item If $ m_0 \not= 0 $, then the extra $ A_{(1)} $ factors in the description of Proposition \ref{prop:description A^} guarantees that $ \widehat{A}_{\pi} $ and $ A_{\pi} $ have the same image under the projection to $ \Sigma_n $. Under this condition, the same argument used in the case $ |O| > 2 $ works to prove that $ x(O) = O' $ without constraints on the cardinality of $ O $.
\end{itemize}
\end{proof}

We observe that for all admissible partition $ \pi $ of $ n $, $ \pi $ is the multiset of the cardinalities of the elements of the orbit set $ \{1,\dots,n\}/A_{\pi} $. This remark together with Lemma \ref{lem:orbits conjugates} yields the following.
\begin{corollary} \label{cor:conjugates B+}
Let $ \pi $ and $ \mu $ be relevant admissible partitions of $ n $.  Let $ x \in \Bgroup_n $. If $ x \widehat{A}_{\pi} x^{-1} \subseteq A_{\mu} $, then $ \pi = \mu $.
\end{corollary}

We are now ready to prove that the $ \widehat{A}_{\pi} $s classify maximal elementary abelian $ 2 $-subgroups of $ \Bgrouppos{n} $ up to conjugation.

\begin{proposition} \label{prop:elementary subgroups}
The following statements are true:
\begin{enumerate}
\item For an admissible partition $ \pi $ of $ n $, $ |A_{\pi}: \widehat{A}_{\pi}| $ is $ 1 $ if $ m_0 = m_1 = 0 $, and it is $ 2 $ otherwise.
\item Every maximal elementary abelian $ 2 $-subgroup of $ \Bgrouppos{n} $ is conjugate in $ \Bgroup_n $ to $ \widehat{A}_{\pi} $ for a unique relevant admissible partition $ \pi $ of $ n $.
Viceversa, $ \widehat{A}_{\pi} $ is a maximal abelian subgroup if $ \pi $ is a relevant admissible partition.
\item If $ \pi $ is a relevant admissible partition of $ n $, then $ N_{\Bgroup_n}(\widehat{A}_{\pi}) = N_{\Bgroup_n}(A_{\pi}) $.
\item For all relevant admissible partitions $ \pi $ of $ n $, $ |N_{\Bgroup_n}(\widehat{A}_{\pi}): N_{\Bgrouppos{n}}(\widehat{A}_{\pi})| $ is equal to $ 1 $ if $ m_0 = m_1 = m_2 = 0 $, and to $ 2 $ otherwise.
\item If $ \pi $ satisfies $ m_0 = m_1 = m_2 = 0 $, then the $ \Bgroup_n $-conjugacy class of $ \widehat{A}_{\pi} $ contains two $ \Bgrouppos{n} $-conjugacy classes represented by $ \widehat{A}_{\pi} $ and $ \widehat{A}_{\pi}^{s_0} $.
Otherwise, the $ \Bgroup_n $- and $ \Bgrouppos{n} $-conjugacy classes of $ \widehat{A}_{\pi} $ coincide.
\end{enumerate}
\end{proposition}
\begin{proof}
\begin{enumerate}
\item It is a direct consequence of Proposition \ref{prop:description A^}.
\item If $ A \leq \Bgrouppos{n} $ is a maximal $ 2 $-abelian subgroup, then it must be contained in a conjugate $ xA_{\pi}x^{-1} $ of $ A_{\pi} $ for some admissible partition $ \pi $ by Theorem \ref{thm:elementary abelian subgroups B}. By maximality $ A = (xA_{\pi}x^{-1}) \cap \Bgrouppos{n} = x \widehat{A}_{\pi}x^{-1} $.
We can exclude irrelevant partitions because $ \widehat{A}_{(1,1)} = \widehat{A}_{(2)} $.

Viceversa, if $ \pi $ is an admissible partition, then by what we just proved $ \widehat{A}_{\pi} $ must be contained in a conjugate of a maximal abelian $ 2 $-subgroup of the form $ A_{\mu} \cap \Bgrouppos{n} = \widehat{A}_{\mu} $ for some $ \mu $ relevant admissible. If $ \pi $ is also relevant, then Corollary \ref{cor:conjugates B+} guarantees that $ \mu = \pi $, so $ \widehat{A}_{\pi} $ is maximal.
This argument shows also that subgroups corresponding to different relevant admissible partitions are not conjugate to each other.
\item Since $ \Bgrouppos{n} $ is normal in $ \Bgroup_n $, then
\[
N_{\Bgroup_n}(A_\pi) = N_{\Bgroup_n}(A_{\pi}) \cap N_{\Bgroup_n}(\Bgrouppos{n}) \subseteq N_{\Bgroup_n}(A_{\pi} \cap \Bgrouppos{n}) = N_{\Bgroup_n}(\widehat{A}_{\pi}).
\]

Let $ x \in N_{\Bgroup_n}(\widehat{A}_{\pi}) $. Then $ x \widehat{A}_{\pi} x^{-1} \subseteq A_{\pi} $, hence by Lemma \ref{lem:orbits conjugates} $ x $ permutes the orbits of $ A_{\pi} $ on $ \{1,\dots,n\} $. This and Theorem \ref{thm:elementary abelian subgroups B} imply that $ x \in \prod_{k=0}^{\infty} (\Bgroup_{2^k} \wr \Sigma_{m_k}) $.
Since by Proposition \ref{prop:description A^} $ \widehat{A}_{\pi} = \widehat{A}_{(1)^{m_0} \sqcup (2)^{m_1}} \times A_{\pi'} $ for some $ \pi' $,
\[
x \in N_{\Bgroup_{m_0+2m_1}}(\widehat{A}_{(1)^{m_0} \sqcup (2)^{m_1}}) \times N_{\Bgroup_{n-m_0-2m_1}}(A_{\pi'}).
\]
By direct inspection, $ A_{(1)} $ and $ A_{(2)} $ are normal in $ \Bgroup_1 $ and $ \Bgroup_2 $, respectively, hence
\[
(\Bgroup_1 \wr \Sigma_{m_0}) \times (\Bgroup_2 \wr \Sigma_{m_1}) \times N_{\Bgroup_{n-m_0-2m_1}}(A_{\pi'}) \subseteq N_{\Bgroup_n}(A_{\pi}).
\]
We deduce that $ N_{\Bgroup_n}(\widehat{A}_{\pi}) \subseteq N_{\Bgroup_n}(A_{\pi}) $.
\item By $ (3) $, $ N_{\Bgrouppos{n}}(\widehat{A}_{\pi}) = N_{\Bgroup_n}(A_{\pi}) \cap \Bgrouppos{n} $, hence the index of the two normalizers is either $ 1 $ or $ 2 $. Thus, it is sufficient to prove that $ N_{\Bgroup_n}(\widehat{A}_{\pi}) \subseteq \Bgrouppos{n} $ if and only if $ m_0 = m_1 = m_2 = 0 $.

By \cite[Proposition 5.3.1]{Swenson}, after identifying $ A_{(2^k)} $ with $ \mathbb{F}_2^{k+1} $, $ N_{\Bgroup_{2^k}}(A_{(2^k)}) $ is the semi-direct product of $ A_{(2^k)} $ and the stabilizer $ G_k $ of the first coordinate vector of $ \mathbb{F}_2^{k+1} $ in $ \Gl_{k+1}(\mathbb{F}_2) $.

We first investigate whether $ G_k $ is contained in the subgroup $ \Bgrouppos{2^k} $.

$ G_k $ is the semi-direct product of the two subgroups
\begin{align*}
H_k &= \{ \left[ \begin{array}{cc} 1 & \underline{b}^T \\
\underline{0} & \mathbb{I}_{k} \end{array} \right]: \underline{b} \in \mathbb{F}_2^k
	\} \\
\tag*{and} K_k &= \{ \left[ \begin{array}{cc} 1 & \underline{0}^T \\
\underline{0} & A \end{array} \right]: A \in \Gl_k(\mathbb{F}_2) \}.
\end{align*}

Each non-trivial element of $ H_k $ acts on $ \mathbb{F}_2^{k+1} $ by fixing $ 2^k $ vectors of $ \mathbb{F}_2^{k+1} $ and flipping the first coordinate bit of the other $ 2^k $. These correspond to products of $ 2^{k-1} $ reflections in $ \Bgroup_{2^k} $. We deduce that $ H \subseteq \Bgrouppos{2^k} $ if and only if $ k \geq 2 $.

The subgroup $ K_k $ stabilizes $ V_k \subseteq A_{(2^k)} $ and is identified with its stabilizer in $ \Sigma_{2^k} $.
It is contained in the alternating group $ A_{2^k} $ if and only if $ k \geq 3 $ (compare \cite[page 21]{Sinha:17}).

We conclude that $ G_k \subseteq \Bgrouppos{2^k} $ if and only if $ k \geq 3 $.

In general, the normalizer $ N_{\Bgroup_n}(A_{\pi}) $ is generated by $ A_{\pi} $ and a product of wreath products of the groups $ G_k $. If $ m_0 = m_1 = m_2 = 0 $, then only $ G_k $ with $ k \geq 3 $ appear in this product, which is thus contained in $ \Bgrouppos{n} $. By $ (1) $, $ A_{\pi} \subseteq \Bgrouppos{n} $, so $ N_{\Bgroup_n}(A_{\pi}) \subseteq \Bgrouppos{n} $.
On the contrary, if at least one among $ m_0 $, $ m_1 $, $ m_2 $ is non-zero, then some $ G_k $ with $ k \leq 2 $ appear in the product, and $ N_{\Bgroup_n}(A_{\pi}) \not\subseteq \Bgrouppos{n} $.

\item It follows directly from $ (4) $ and Proposition \ref{prop:conjugacy classes}.
\end{enumerate}
\end{proof}

\subsection{The Weil invariant subalgebras of the cohomology of elementary abelian \texorpdfstring{$ 2 $}{2}-subgroups}

In this subsection we compute the subalgebra of the cohomology of maximal abelian $ 2 $-subgroups of $ \Bgrouppos{n} $ constisting of the invariant elements with respect to the conjugation action of their normalizer. This is a relatively straightforward, but lengthy, calculation that uses Galois theory. Therefore, we split it into separate lemmas.

As a preliminary step, we recall a known general property of invariant rings.

\begin{proposition}[Proposition 1.2.4 of \cite{Smith}] \label{prop:general Galois}
Suppose that $ V $ is a finite dimensional faithful representation of a finite group $ G $ over a field $ k $. Then $ k(V) $ is a Galois extension of $ k(V)^G $ with Galois group $ G $. The field $ k(V)^G $ is the field of fractions of $ k[V]^G $, and $ k[V]^G $ is integrally closed.
\end{proposition}

\begin{corollary} \label{cor:Galois}
Let $ R = k[x_1,\dots,x_m] $ be a polynomial ring over a field $ k $. The field of fraction $ \Quot(R^{\otimes n}) $ is a Galois extension of $ \Quot((R^{\otimes n})^{\Sigma_n}) $ with Galois group isomorphic to $ \Sigma_n $.
Moreover, $ (R^{\otimes n})^{\Sigma_n} $ is integrally closed.
\end{corollary}

We now compute the invariants in the cohomology of $ \widehat{A}_{(4)} $, that will be our basic building block to construct invariants in the general case. We recall that, by Proposition \ref{prop:conjugacy classes} $(1)$, $ \widehat{A}_{(4)} = A_{(4)} = C_2 \times V_2 $. We will use the identification $ H^*(\widehat{A}_{(4)}; \mathbb{F}_2) = \mathbb{F}_2[x,y_1,y_2] $ described in Remark \ref{rem:invariants}.

We also recall that if $ A \leq G $ is a subgroup, the action of $ N_G(A) $ on $ H^*(A) $ is known to factor through the Weyl group $ W_{G}(A) = N_{G}(A)/A $ (see \cite[Lemma II.3.1]{Adem-Milgram}).

We suggest that the reader compares our calculations below with Theorem III.1.3 of \cite{Adem-Milgram}.

\begin{lemma} \label{lem:A4 Galois}
Let $ h_3 = y_1^3 + y_1^2y_2 + y_2^3 \in H^*(\widehat{A}_{(4)}; \mathbb{F}_2) = k[x,y_1,y_2] $. 
The field extension $ k(x,y_1,y_2)/k(f_4,d_2,d_3)[h_3] $ is a Galois extension with Galois group $ W_{\Bgrouppos{4}}(\widehat{A}_{(4)}) $.
\end{lemma}
\begin{proof}
To fix notation, we let $ S = H^*(A_{(4)}; \mathbb{F}_2) = \mathbb{F}[x,y_1,y_2] $, $ S' = \mathbb{F}_2[f_4,y_1,y_2] $, $ S''' = \mathbb{F}_2[f_4,d_2,d_3] $, and $ S'' $ the subring of $ S $ generated by $ S''' $ and $ h_3 $. There are subring inclusions $ S''' \leq S'' \leq S' \leq S $.

$ W_{\Bgrouppos{4}}(A_{(4)}) $ is a semi-direct product of the subgroups $ H_2 \cong (\mathbb{Z}/2\mathbb{Z})^2 $ and $ K_2 \cap \Bgrouppos{4} \cong\mathbb{Z}/3\mathbb{Z} $ introduced in the proof of point $ (4) $ of Proposition \ref{prop:conjugacy classes}.
The four elements of $ H_2 $ fix $ y_1, y_2 $ and map $ x $ to $ x, x+y_1, x+y_2, x+y_1+y_2 $ respectively.
$ \Quot(S) $ is generated over $ \Quot(S') $ by $ x $, which is a root of the polynomial
\[
p(t) = (t-x)(t-x-y_1)(t-x-y_2)(t-x-y_1-y_2) = t^4 + d_2t^2 + f_4 \in S'[t].
\]
Therefore, $ \Quot(S) $ is the splitting field of $ p(t) $ over $ \Quot(S') $ and the extension is Galois.
Since the difference between any two roots of $ p(t) $ is in $ S' $, the action of an element of $ \Gal(\Quot(S)/\Quot(S')) $ is determined by its value on $ x $. We deduce that $ |\Gal(\Quot(S)/\Quot(S'))| \leq \deg(p(t)) = 4 $. However, the Galois group contains $ H_2 $ as a subgroup, so they must be equal by reason of cardinality.

The action of $ K_2 \cap \Bgrouppos{4} \cong \mathbb{Z}/3\mathbb{Z} $ on $ S' $ fixes $ f_4 $ and act on $ \mathbb{F}_2[y_1,y_2] $ by cyclically permuting $ y_1 $, $ y_2 $ and $ y_1 + y_2 $.
$ h_3(d_3-h_3) + d_2^3 + d_3^2 = 0 $, hence $ h_3 $ is integral on $ S^{W_{\Bgroup_4}(A_{(4)})} \cong \mathbb{F}_2[f_4,d_2,d_3] $ which, as a polynomial algebra, is integrally close. 
Therefore $ h_3 \notin \Quot(S''') $ and $ \Quot(S'')/\Quot(S''') $ is a Galois field extension of degree $ 2 $.
By an argument of Wilkerson \cite[page 423]{Wilkerson} $ \Quot(S')/\Quot(S''') $ is a Galois extension with Galois group $ \Gal(\Quot(S')/\Quot(S''') = K \cong \Gl_2(\mathbb{F}_2) $.
Thus $ \Quot(S')/\Quot(S'') $ is Galois with $ \Gal(\Quot(S')/\Quot(S'') = K \cap \Bgrouppos{4} $.
We deduce that $ W_{\Bgrouppos{4}}(\widehat{A}_{(4)}) $ coincides with $ \Gal(\Quot(S)/\Quot(S'')) $.
\end{proof}

\begin{corollary} \label{cor:invariants A4}
There is an isomorphism
\[
\left[ H^*(A_{(4)}; \mathbb{F}_2)\right]^{N_{\Bgrouppos{4}}(A_{(4)})} \cong \frac{\mathbb{F}_2[f_4,d_2,d_3,h_3]}{(h_3(d_3-h_3) + d_2^3 + d_3^2)},
\]
where $ f_4 $, $ d_2 $ and $ d_3 $ are as in Theorem \ref{thm:elementary abelian subgroups B} and $ h_3 = y_1^3 + y_1^2y_2 + y_2^3 $.
\end{corollary}
\begin{proof}
The relation $ h_3(d_3-h_3) + d_2^3 + d_3^2 = 0 $ has been established in the proof above.
By Galois theory, Lemma \ref{lem:A4 Galois} and Proposition \ref{prop:general Galois}
\[
k[x,y_1,y_2]^{N_{\Bgrouppos{4}}(\widehat{A}_{(4)})} = \frac{k(f_4,d_2,d_3)[h_3]}{(h_3(d_3-h_3) + d_2^3 + d_3^2)}.
\]
Both sides are integrally closed, so by intersecting them with $ k[x,y_1,y_2] $ we obtain the desired identity.

We suggest comparison of our argument with the more general result \cite[Proposition 3.11.2]{Derksen-Kemper}.
\end{proof}

Next, we study the invariant subalgebra of the cohomology of $ \widehat{A}_{(4)^m} $, where $(4)^m $ is the partition of $ 4m $ containing $ 4 $ with arbitrary multiplicity $ m $.

To simplify notation, we provide the following definition.
\begin{definition} \label{def:h}
Let $ m \in \mathbb{N} $. Let $ f_4,d_2,d_3,h_3 $ be as in Theorem \ref{thm:elementary abelian subgroups B} and Corollary \ref{cor:invariants A4}. Let $ h_3^\perp = d_3 - h_3 $. For $ 1 \leq i \leq m $, we define the elements
\begin{align*}
	h_{3,i} &= 1^{\otimes (i-1)} \otimes h_3 \otimes 1^{\otimes (m-i)} \in \mathbb{F}_2[x,y_1,y_2]^{\otimes m} = H^*(\widehat{A}_{(4)^m}; \mathbb{F}_2), \\
	h_{3,i}^{\perp} &= 1^{\otimes (i-1)} \otimes h_3^{\perp} \otimes 1^{\otimes (m-i)} \in \mathbb{F}_2[x,y_1,y_2]^{\otimes m} = H^*(\widehat{A}_{(4)^m}; \mathbb{F}_2).
\end{align*}
\begin{align*}
\tag*{\textit{Let}} h_{3m} &= \sum_{\substack{S \subseteq \{1,\dots,m\}\\ |S| \mbox{ even}}} \prod_{i \notin S} h_{3,i} \prod_{i \in S} h_{3,i}^{\perp} \\
\tag*{\textit{and}} h_{3m}^{\perp} &= \sum_{\substack{S \subseteq \{1,\dots,m\}\\ |S| \mbox{ even}}} \prod_{i \notin S} h_{3,i}^{\perp} \prod_{i \in S} h_{3,i} \\
\end{align*}
\end{definition}

The following statement is proved by a direct induction argument.
\begin{lemma} \label{lem:cup h}
	In $ H^*(\widehat{A}_{(2)^m}; \mathbb{F}_2) $,
	\[
	h_{3m} \cdot (d_3^{\otimes m} - h_{3m}) = \left\{ \begin{array}{ll}
		{d_3^2}^{\otimes m} + \sum_{i=1}^m (d_3^2)^{\otimes i-1} \otimes d_2^3 \otimes (d_3^2)^{\otimes m-i} & \mbox{if } m \mbox{ is odd} \\
		\sum_{i=1}^m (d_3^2)^{\otimes i-1} \otimes d_2^3 \otimes (d_3^2)^{\otimes m-i} & \mbox{if } m \mbox{ is even}
	\end{array} \right.
	\]
\end{lemma}

\begin{lemma} \label{lem:A4m Galois}
Let $ k' $ be the sub-field of $ k(x,y_1,y_2)^{\otimes m} $ generated by the invariants $ \Quot(k[f_4,d_2,d_3])^{\Sigma_m} $ and $ h_{3m} $. The field extension $ k(x,y_1,y_2)^{\otimes m}/k' $ is a Galois extension with Galois group  $ W_{\Bgrouppos{4m}}(\widehat{A}_{(4)^m}) $.
\end{lemma}
\begin{proof}
To simplify notation, we write $ T = k[x,y_1,y_2]^{\otimes m} $, $ T' = k[f_4,d_2,d_3]^{\otimes m} $, $ T''' = (k[f_4,d_2,d_3]^{\otimes m})^{\Sigma_m} $ and $ T'' $ for the subring of $ T''' $ generated by $ T''' $ and $ h_{3m} $.

By Lemma \ref{lem:A4 Galois} and Corollary \ref{cor:Galois}, $ \Quot(T)/\Quot(T') $ and $ \Quot(T')/\Quot(T''') $ are Galois extensions with Galois groups $ W_{\Bgroup_{4}}(\widehat{A}_{(4)})^m $ and $ \Sigma_m $, respectively.
The Weil group $ W_{\Bgroup_{4m}}(\widehat{A}_{(4)^m}) $ is clearly a subgroup of $ \Gal(\Quot(T)/\Quot(T''')) $. By cardinality reasons, these two groups are equal.

By Lemma \ref{lem:cup h}, $ h_{3m} $ is integral over $ \Quot(T) $. 
By Corollary \ref{cor:Galois}, $ h_{3m} \notin \Quot(T''') $, hence the extension $ \Quot(T'')/\Quot(T''') $ has degree $ 2 $. In particular, it is Galois. Since $ h_{3m} $ is $ W_{\Bgrouppos{4m}}(A_{(4)^{m}}) $-invariant, $ \Quot(T)/\Quot(T'') $ is Galois with Galois group $ W_{\Bgrouppos{4m}}(A_{(4)^{m}}) $.
\end{proof}

We deduce the invariant subring with essentially the same argument used in the proof of Corollary \ref{cor:invariants A4}.
\begin{corollary} \label{cor:invariants A4m}
Let $ m \in \mathbb{N} $. There is an isomorphism
\[
\left[ H^*(\widehat{A}_{(4)^m}; \mathbb{F}_2) \right]^{N_{\Bgrouppos{4m}}(\widehat{A}_{(4)^m})} \cong \frac{(\mathbb{F}_2[f_4,d_2,d_3]^{\otimes m})^{\Sigma_m}[h_{3m}]}{I_m},
\]
where $ I_m $ is the ideal generated by the relation in the statement of Lemma \ref{lem:cup h}.
\end{corollary}

The general case is described below.
\begin{theorem} \label{thm:invariants normalizer}
Let $ \pi $ be a relevant admissible partition of $ n $.
With reference to the statement of Theorem \ref{thm:elementary abelian subgroups B}, if $ m_2 = 0 $ or $ (m_0,m_1) \not= (0,0) $, then there is an isomorphism
\begin{gather*}
	\left[ H^*(\widehat{A}_{\pi}; \mathbb{F}_2) \right] ^{N_{\Bgrouppos{n}}(\widehat{A}_{\pi})} \cong \frac{ \bigotimes_{k \in \mathbb{N}} \left[ \mathbb{F}_2[f_{2^k},d_{2^k-1},\dots,d_{2^k-2^{k-1}}]^{\otimes m_k} \right]^{\Sigma_{m_k}}}{(e_1(f_1) + e_1(d_1))},\\
	\tag*{where} e_1(f_1) = \sum_{i=1}^{m_0} 1^{\otimes i-1} \otimes f_1 \otimes 1^{\otimes m_0-i} \otimes 1_{\bigotimes_{k \geq 1} \mathbb{F}_2[f_{2^k},\dots,d_{2^{k-1}}]^{\otimes m_k}}\\ \tag*{and} e_1(d_1) = \sum_{j=1}^{m_1} 1_{\mathbb{F}_2[f_0]^{\otimes m_0}} \otimes 1^{\otimes j-1} \otimes d_1 \otimes 1^{\otimes m_1 - j} \otimes 1_{\bigotimes_{k \geq 2} \mathbb{F}_2[f_{2^k},\dots,d_{2^k-1}]^{\otimes m_k}}.
\end{gather*}
If $ m_0 = m_1 = 0 $ and $ m_2 \not= 0 $, then $ \left[ H^*(\widehat{A}_{\pi}; \mathbb{F}_2) \right] ^{N_{\Bgrouppos{n}}(\widehat{A}_{\pi})} $ is isomorphic to
\[
\frac{\bigotimes_{k \in \mathbb{N}} \left[ \mathbb{F}_2[f_{2^k},d_{2^k-1},\dots,d_{2^k-2^{k-1}}]^{\otimes m_k} \right]^{\Sigma_{m_k}}[h_{3m} \otimes 1_{\bigotimes_{k=3}^\infty \mathbb{F}_2[f_{2^k},\dots,d_{2^{k-1}}]^{\otimes m_k}}]}{I_{\pi}},
\]
where $ I_{\pi} $ is the ideal generated by $ I_m \otimes 1_{\bigotimes_{k=3}^{\infty} \mathbb{F}_2[f_{2^k},\dots,d_{2^{k-1}}]^{\otimes m_k}} $.
In particular, it is a free module of rank $ 2 $ over
\[
\bigotimes_{k \in \mathbb{N}} \left[ \mathbb{F}_2[f_{2^k},d_{2^k-1},\dots,d_{2^k-2^{k-1}}]^{\otimes m_k} \right]^{\Sigma_{m_k}}
\]
with basis given by the couple $ \{ 1, h_{3m_2} \otimes 1_{\bigotimes_{k=3}^{\infty} \mathbb{F}_2[f_{2^k},\dots,d_{2^{k-1}}]^{\otimes m_k}}\} $.

Moreover, via the isomorphisms above, the restriction map
\[
\left[ H^*(A_{\pi}; \mathbb{F}_2) \right]^{N_{\Bgroup_n(A_{\pi})}} \to \left[ H^*(\widehat{A}_{\pi}; \mathbb{F}_2) \right]^{N_{\Bgrouppos{n}(\widehat{A}_{\pi})}}
\]
is identified with the obvious quotient map (if $ m_2 = 0 $ or $ (m_0,m_1) \not= (0,0) $) or inclusion map (if $ (m_0,m_1) = (0,0) $ and $ m_2 \not= 0 $).
\end{theorem}
\begin{proof}
By Proposition \ref{prop:description A^}
\[
H^*(\widehat{A}_{\pi};\mathbb{F}_2) \cong \frac{\Sym(A_{\pi}^{\vee})}{(\phi_{m_0,m_1})}.
\]

$ \phi_{m_0,m_1} $ extend to a homomorphism $ \Bgroup_n \to \mathbb{F}_2 $, so it is invariant by the action of the normalizer $ N_{\Bgroup_n}(\widehat{A}_{\pi}) $. Moreover, by Proposition \ref{prop:conjugacy classes} $ (3) $ the normalizers of $ \widehat{A}_{\pi} $ and $ A_{\pi} $ coincide.
Consequently
\[
\left[ H^*(\widehat{A}_{\pi}; \mathbb{F}_2) \right]^{N_{\Bgroup_n}(\widehat{A}_{\pi})} \cong \frac{\left[ H^*(A_{\pi}; \mathbb{F}_2) \right]^{N_{\Bgroup_n}(A_{\pi})}}{(\phi_{m_0,m_1})}.
\]

By points $ (1) $ and $ (4) $ of Proposition \ref{prop:conjugacy classes}, if $ m_0 = m_1 = m_2 = 0 $ or $ (m_0,m_1) \not= (0,0) $, then $ W_{\Bgrouppos{n}}(\widehat{A}_{\pi}) \cong W_{\Bgroup_{n}}(A_{\pi}) $. Therefore, in such cases,
\[
\left[ H^*(\widehat{A}_{\pi}; \mathbb{F}_2) \right]^{N_{\Bgrouppos{n}}(\widehat{A}_{\pi})} \cong \frac{\left[ H^*(A_{\pi}; \mathbb{F}_2) \right]^{N_{\Bgroup_n}(A_{\pi})}}{(\phi_{m_0,m_1})}
\]
and the desired statement follows by identifying $ \phi_{m_0,m_1} $  with $ e(f_1) + e(d_1) $.

On the contrary, if $ m_0 = m_1 = 0 $ and $ m_2 \not= 0 $, then $ W_{\Bgrouppos{n}}(\widehat{A}_{\pi}) $ is isomorphic to a subgroup of index $ 2 $ in $ W_{\Bgroup_n}(A_{\pi}) $. By Proposition \ref{prop:description A^}, Proposition \ref{prop:conjugacy classes} $ (3) $, and \cite[Proposition 5.3.1]{Swenson}, $ \widehat{A}_{\pi} $ is the product of $ \widehat{A}_{(4)^{m_2}} $ and another maximal elementary abelian $ 2 $-subgroup encompassed by the previous analysis, and the invariant subalgebra of its cohomology splits as the product of the invariant subalgebras of the factors.
Therefore, we reduce to the special case $ \pi = (4)^{m_2} $ and the theorem follows from Corollary \ref{cor:invariants A4m}.
\end{proof}

\subsection{The interplay between the almost-Hopf ring structure and the restrictions to elementary abelian \texorpdfstring{$ 2 $}{2}-subgroups}

In this sub-sub-section, we investigate the link between elementary abelian $ 2 $-subgroups and the almost-Hopf ring structural morphisms. For every relevant admissible partition $ \pi $ of $ n $, we let $ \rho_\pi \colon H^*(\Bgrouppos{n}; \mathbb{F}_2) \to H^*(\widehat{A}_{\pi}; \mathbb{F}_2) $ be the restriction map.

The link with the cup product is classically well-known.
\begin{proposition} \label{prop:rho cup}
$ \rho_{\pi} $ preserves cup products for all $ \pi $.
\end{proposition}

The relation between the maps $ \rho_\pi $ and the coproduct in $ \AB $ is contained in the following easily proved statements.
\begin{proposition} \label{prop:coproduct cup elementary subgroups}
Let $ \pi $ and $ \pi' $ be relevant admissible partitions of $ n $ and $ n' $, respectively. Let $ m_k $ and $ m_k' $ be the multiplicities of $ 2^k $ in $ \pi $ and $ \pi' $, respectively.
\begin{enumerate}
\item With reference to the isomorphism of Theorem \ref{thm:invariants normalizer}, the inclusion of invariants induces a homomorphism
\begin{gather*}
\nu \colon \frac{\bigotimes_{k \in \mathbb{N}} [ \mathbb{F}_2[f_{2^k},d_{2^k-1},\dots,d_{2^k-2^{k-1}}]^{\otimes (m_k+m_k')}]^{\Sigma_{m_k+m_k'}}}{(e_1(f_1) + e_1(d_1))} \to \\
\frac{\bigotimes_{k \in \mathbb{N}} [ \mathbb{F}_2[f_{2^k},d_{2^k-1},\dots,d_{2^k-2^{k-1}}]^{\otimes m_k}]^{\Sigma_{m_k}}}{(e_1(f_1) + e_1(d_1))} \otimes \frac{\bigotimes_{k \in \mathbb{N}} [ \mathbb{F}_2[f_{2^k},d_{2^k-1},\dots,d_{2^k-2^{k-1}}]^{\otimes m_k'}]^{\Sigma_{m_k'}}}{(e_1(f_1) + e_1(d_1))}
\end{gather*}
by passing to quotients.
\item There is a commutative diagram
\begin{center}
\begin{tikzcd}
H^*(\Bgrouppos{n+m}; \mathbb{F}_2) \arrow{r}{\Delta_{(n,m)}} \arrow{d}{\rho_{{\pi \sqcup \pi'}}} & H^*(\Bgrouppos{n}; \mathbb{F}_2) \otimes H^*(\Bgrouppos{m}; \mathbb{F}_2) \arrow{d}{\rho_{\pi} \otimes \rho_{\pi'}} \\
\left[ H^*(\widehat{A}_{\pi \sqcup \pi'}) \right]^{N_{\Bgrouppos{n+m}(\widehat{A}_{\pi \sqcup \pi'})}} \arrow{r}{\nu'} & \left[ H^*(\widehat{A}_{\pi}) \right]^{N_{\Bgrouppos{n}(\widehat{A}_{\pi})}} \otimes \left[ H^*(\widehat{A}_{\pi'}) \right]^{N_{\Bgrouppos{m}(\widehat{A}_{\pi'})}},
\end{tikzcd}
\end{center}
where $ \nu' = \nu $ if $ (m_0+m_0',m_1+m_1') = (0,0) $ or $ m_2+m_2' = 0 $, and
\[
\nu'(a+b(h_{3(m_2+m_2')}\otimes 1)) = \nu(a) + \nu(b) \left((h_{3m_2} \otimes 1) \otimes (h_{3m_2'} \otimes 1) + (h^\perp_{3m_2} \otimes 1) \otimes (h^{\perp}_{3m_2'} \otimes 1)\right)
\]
for all $ a,b \in \frac{\bigotimes_{k \in \mathbb{N}} [ \mathbb{F}_2[f_{2^k},d_{2^k-1},\dots,d_{2^k-2^{k-1}}]^{\otimes (m_k+m_k')}]^{\Sigma_{m_k+m_k'}}}{(e_1(f_1) + e_1(d_1))} $ if $ m_0 = m_0' = m_1 = m_1' = 0 $ and $ m_2 + m_2' \not= 0 $.
\end{enumerate}
\end{proposition}
\begin{proof}
(1) is proved by direct inspection: the map defined by restriction of invariants exists without taking the quotient by $ e_1(f_1) + e_1(d_1) $, and maps $ e_1(f_1) + e_1(d_1) $ to $ (e_1(f_1) + e_1(d_1)) \otimes 1 + 1 \otimes (e_1(f_1) + e_1(d_1)) $.

Regarding (2), as $ \Delta_{(n,m)} $, $ \rho_{\pi \sqcup \pi'} $, $ \rho_{\pi} $ and $ \rho_{\pi'} $ are all restriction maps, the morphism $ \nu' $ induced by restriction of invariants makes the diagram commute.
If $ (m_0+m_0',m_1+m_1') \not= (0,0) $ or $ m_2 + m_2' = 0 $ $ \nu' = \nu $ by construction. Otherwise, $ \nu' = \nu $ on
\[
\bigotimes_{k \in \mathbb{N}} [ \mathbb{F}_2[f_{2^k},d_{2^k-1},\dots,d_{2^k-2^{k-1}}]^{\otimes (m_k+m_k')}]^{\Sigma_{m_k+m_k'}}.
\]
Therefore, by Theorem \ref{thm:invariants normalizer} and Proposition \ref{prop:rho cup}, it is enough to prove that $ h_{3(m_2+m_2')} \otimes 1 $ restricts to $ (h_{3m_2} \otimes 1) \otimes (h_{3m_2'} \otimes 1) + (h^\perp_{3m_2} \otimes 1) \otimes (h^{\perp}_{3m_2'} \otimes 1) $ in $ \left[ H^*(\widehat{A}_{\pi}) \right]^{N_{\Bgrouppos{n}(\widehat{A}_{\pi})}} \otimes \left[ H^*(\widehat{A}_{\pi'}) \right]^{N_{\Bgrouppos{m}(\widehat{A}_{\pi'})}} $, which is clear from its definition.
\end{proof}

In contrast, the relation between transfer product and restriction to the subgroups $ \widehat{A}_{\pi} $ is non-trivial. It requires the classical Cartan--Eilenberg double-coset formula, that we recall below.
\begin{theorem}[{\cite[Theorem II.6.2]{Adem-Milgram}}] \label{thm:Cartan-Eilenberg}
Let $ H,K $ be subgroup of a finite group $ G $. Let $ \mathcal{R} $ be a set of representatibes for the double cosets $ H\backslash G/K $. Let $ c_r $ be the conjugation homomorphism. Then
\[
\res^G_H \circ \tr^G_K = \sum_{r \in \mathcal{R}} \tr^H_{H \cap rKr^{-1}} \circ c_r^* \circ \res^K_{r^{-1}Hr \cap K} \colon H^*(K) \to H^*(H).
\]
\end{theorem}

To effectively apply Theorem \ref{thm:Cartan-Eilenberg} for our purpose, we need a technical combinatorial lemma.

\begin{lemma} \label{lem:double cosets}
Let $ n,m \in \mathbb{N} $ and let $ \pi $ be ad admissible partition of $ n+m $.
\begin{enumerate}
\item Let $ X = \{x \in \Bgroup_{n+m}: x^{-1} A_{\pi} x \subseteq \Bgroup_n \times \Bgroup_m \} $.
Then there are a left $ A_{\pi} $-action and a right $ (\Bgroup_n \times \Bgroup_m) $-action on $ X $, commuting with each other, given by left and right multiplication, respectively.
Moreover, The double quotient $ A\backslash X/B $ is in bijective correspondence with the set of triples $ (\pi',\pi'', \{\sigma_k\}_{k=0}^\infty) $, where:
\begin{itemize}
\item $ \pi' $ and $ \pi'' $ are admissible partitions of $ n $ and $ m $, respectively, such that $ \pi = \pi' \sqcup \pi'' $
\item for all $ k \geq 0 $, $ \sigma_k $ is a $ (m_k(\pi'),m_k(\pi'')) $-shuffle.
\end{itemize}
\item Let $ X' = \{ x \in \Bgrouppos{n}: x^{-1} \widehat{A}_{\pi}x \subseteq \Bgrouppos{n} \times \Bgrouppos{m} \} $.
Then there are a left $ \widehat{A}_{\pi} $-action and a right $ (\Bgrouppos{n} \times \Bgrouppos{m}) $-action on $ X' $, commuting with each other, given by left and right multiplication, respectively.
Moreover, if $ \pi $ is relevant, there is a function $ t \colon \widehat{A}_{\pi} \backslash X / (\Bgrouppos{n} \times \Bgrouppos{m}) \to A_{\pi} \backslash X / (\Bgroup_n \times \Bgroup_m) $ given by inclusion of double orbits. The image of $ t $ comprises those elements corresponding to triples $ (\pi',\pi'',\{\sigma_k\}_k) $ such that at least one between $ \pi' $ and $ \pi'' $ has $ (m_0,m_1) = (0,0) $.
Moreover, every non-empty fiber of $ t $ consists of two elements that correspond to each other under right multiplication by $ s_0 \times s_0 \in \Bgroup_n \times \Bgroup_m $.
\end{enumerate}
\end{lemma}
\begin{proof}
\begin{enumerate}
\item The existence of the two actions is trivial.
We observe that $ x \in X $ if and only if $ x(\{1,\dots,n\}) $ and $ x(\{n+1,\dots,n+m\}) $ are union of orbits in $ \{1,\dots, n+m\}/A_{\pi} $.
Elements of the double quotient $ A_{\pi} \backslash X / (\Bgroup_n \times \Bgroup_m) $ only record which orbits of $ \{1,\dots, n+m\}/A_{\pi} $ are contained in $ x(\{1,\dots,n\}) $.
The multiplicity $ m_k(\pi') $ (respectively $ m_k(\pi'') $) counts how many orbits of cardinality $ 2^k $ are contained in $ x(\{1,\dots,n\}) $ (respectively how many are not).
The shuffle $ \sigma_k $ detects exactly which of these orbits are contained.
\item The existence of the two actions is straightforward. We first assume that $ (m_0(\pi),m_1(\pi)) \not= (0,1) $. In this case $ \widehat{A}_{\pi} $ and $ A_{\pi} $ have the same orbit space on $ \{1,\dots,n\} $, so $ X' \subseteq X $ and $ x^{-1} \widehat{A}_{\pi} x \subseteq \Bgrouppos{n} \times \Bgrouppos{m} $ implies $ x^{-1} A_{\pi} x \subseteq \Bgroup_n \times \Bgroup_m $ and $ t $ is well-defined.
If $ x \in X $ corresponds to $ (\pi',\pi'',\{\sigma_k\}_k) $, then $ (x^{-1} \widehat{A}_{\pi} x) \cap (\Bgrouppos{n} \times \Bgrouppos{m}) $ is conjugate to $ \widehat{A}_{\pi'} \times \widehat{A}_{\pi''} $. Therefore, $ x \in X' $ if an only if $ \widehat{A}_{\pi} $ and $ \widehat{A}_{\pi'} \times \widehat{A}_{\pi''} $ have the same cardinality. By proposition \ref{prop:conjugacy classes}(1), this happens if and only if at least one between $ \pi' $ and $ \pi'' $ has $ (m_0,m_1) = (0,0) $.

For all $ x \in X' $, $ \Stab_{A_{\pi} \times (\Bgroup_n \times \Bgroup_m)}(x) = \{(a,x^{-1}ax): a \in A_{\pi} \} \cong A_{\pi} $,
and similarly $ \Stab_{\widehat{A}_{\pi} \times (\Bgrouppos{n} \times \Bgrouppos{m})}(x) \cong \widehat{A}_{\pi} $.
Since $ |\Stab_{A_{\pi} \times (\Bgroup_n \times \Bgroup_m)}(x)| $ is twice $ |\Stab_{\widehat{A}_{\pi} \times (\Bgrouppos{n} \times \Bgrouppos{m})}(x)| $ and $ |A_{\pi} \times (\Bgroup_n \times \Bgroup_m)| $ is eight times $ |\widehat{A}_{\pi} \times (\Bgrouppos{n} \times \Bgrouppos{m})| $, the orbit of $ x $ under the $ (A_{\pi},\Bgroup_n \times \Bgroup_m) $-action contains $ 4 $ orbits under the $ (\widehat{A}_{\pi},\Bgrouppos{n} \times \Bgrouppos{m}) $-action. If $ x \in X' $, then they are represented by $ x $, $ x(s_0,s_0) $, $ x(s_0 \times 1_{\Bgroup_m}) $ and $ x(1_{\Bgroup_n} \times s_0) $ and only the first two belong to $ X' $.

In the exceptional case $ (m_0(\pi),m_1(\pi)) = (0,1) $, if might happen that $ x(\widehat{A}_{\pi})x^{-1} \subseteq \Bgroup_{n} \times \Bgroup_{m} $ even if the corresponding couple of partitions $ (\pi',\pi'') $ does not satisfy $ \pi = \pi' \sqcup \pi'' $. However, in all those extra cases $ x(\widehat{A}_{\pi})x^{-1} \not\subseteq \Bgrouppos{n} \times \Bgrouppos{m} $, hence the previous argument works in general.
\end{enumerate}
\end{proof}

We are not ready to investigate the relation between the transfer product and the maps $ \rho_{\pi} $. As a notational convention, given an admissible partition $ \pi $ of $ n $, we let $ \rho'_{\pi} \colon H^*(\Bgrouppos{n}; \mathbb{F}_2) \to H^*(\widehat{A}_{\pi}; \mathbb{F}_2) $ be the composition of the map $ c_s^* $ induced by conjugation by a reflection $ s \in \Bgroup_n $ and the restriction map to $ \widehat{A}_{\pi}^s $. $ \rho'_{\pi} $ does not depend on the choice of $ s $ because it can be equivalently written as $ \rho \circ \iota $.

\begin{proposition} \label{prop:transfer elementary subgroups}
The following statements are true:
\begin{enumerate}
\item Let $ \pi $ be a relevant admissible partition of $ n $. If $ (m_0(\pi),m_1(\pi)) \not= (0,0) $, then $ \rho'_{\pi} = \rho_{\pi} $.
\item Let $ \pi $ be a relevant admissible partition of $ n $. If $ m_0(\pi) = m_1(\pi) = 0 $ but $ m_2(\pi) \not= 0 $, then $ \rho'_{\pi} = S \rho_{\pi} $, where
\[
S \colon \left[ H^*(\widehat{A}_{\pi}; \mathbb{F}_2) \right]^{N_{\Bgrouppos{n}}(\widehat{A}_{\pi})} \to \left[ H^*(\widehat{A}_{\pi}; \mathbb{F}_2) \right]^{N_{\Bgrouppos{n}}(\widehat{A}_{\pi})}
\]
fixes $ \bigotimes_{k \in \mathbb{N}} [ \mathbb{F}_2[f_{2^k},d_{2^k-1},\dots,d_{2^k-2^{k-1}}]^{\otimes m_k(\pi)}]^{\Sigma_{m_k(\pi)}} $ and maps $ h_{3m_2(\pi)} \otimes 1 $ to $ h_{3m_2(\pi)}^{\perp} \otimes 1 $.
\item Let $ \pi' $ and $ \pi'' $ be admissible partitions of $ n $ and $ m $, respectively. There is a map
\begin{gather*}
\tau_{\pi',\pi''} \colon (\id + S \otimes S) \left(\left[ H^*(\widehat{A}_{\pi'}; \mathbb{F}_2) \right]^{N_{\Bgrouppos{n}}(\widehat{A}_{\pi'})} \times \left[ H^*(\widehat{A}_{\pi''}; \mathbb{F}_2) \right]^{N_{\Bgrouppos{m}}(\widehat{A}_{\pi''})} \right) \\
\to \left[ H^*(\widehat{A}_{\pi' \sqcup \pi''}) \right]^{N_{\Bgrouppos{n+m}}(\widehat{A}_{\pi'\sqcup \pi''})}
\end{gather*}
induced by the symmetrization over $ \prod_{k=0}^\infty \Sh(m_k,m_k') \subseteq \prod_{k=0}^\infty \Sigma_{m_k+m_k'} $.
\item For all $ n,m \in \mathbb{N} $ and for all relevant admissible partition $ \pi $ of $ n+m $, the following equality of morphisms $ H^*(\Bgrouppos{n}; \mathbb{F}_2) \otimes H^*(\Bgrouppos{m}; \mathbb{F}_2) \to H^*(\widehat{A}_{\pi}; \mathbb{F}_2) $ holds:
\[
\rho_{\pi} \circ \odot = \sum_{(\pi',\pi'')} \tau_{\pi',\pi''} \circ  (\rho_{\pi'} \otimes \rho_{\pi''} + \rho'_{\pi'} \otimes \rho'_{\pi''}),
\]
where the sum is over couples of relevant admissible partitions $ (\pi',\pi'') $ such that $ \pi' \sqcup \pi'' = \pi $ and at least one between $ \pi' $ and $ \pi'' $ satisfies $ (m_0,m_1) = (0,0) $.
\end{enumerate}
\end{proposition}
\begin{proof}
\begin{enumerate}
\item If $ (m_0(\pi),m_1(\pi)) \not= (0,0) $, then $ A_{\pi} $ contains a reflection. Thus, we can assume that $ s \in A_{\pi} $. The conjugation by $ s $ is trivial on $ \widehat{A}_{\pi} $ because $ A_{\pi} $ is abelian, hence $ \widehat{A}_{\pi}^s = \widehat{A}_{\pi} $ and $ c_s^* = \id $.
\item If $ m_0(\pi) = m_1(\pi) = 0 $ but $ m_2(\pi) \not= 0 $, then $ N_{\Bgroup_n}(\widehat{A}_{\pi}) $ contains a reflection $ s $. Then $ \widehat{A}_{\pi}^s = \widehat{A}_{\pi} $, and $ c_s^* $ exchanges $ h_{3m_2(\pi)} \otimes 1 $ and $ h_{3m_2(\pi)}^\perp \otimes 1 $.
\item If $ (m_0(\pi'),m_1(\pi')) \not= (0,0) $ or $ m_2(\pi') = 0 $ or $ (m_0(\pi''),m_1(\pi'')) \not= (0,0) $ or $ m_2(\pi'') = 0 $, then the domain of $ \tau_{\pi',\pi''} $ is 
\[
\frac{\bigotimes_{k \in \mathbb{N}} [ \mathbb{F}_2[f_{2^k},\dots,d_{2^k-2^{k-1}}]^{\otimes m_k(\pi')}]^{\Sigma_{m_k(\pi')}}}{(e_1(f_1) + e_1(d_1))} \otimes \frac{\bigotimes_{k \in \mathbb{N}} [ \mathbb{F}_2[f_{2^k},\dots,d_{2^k-2^{k-1}}]^{\otimes m_k(\pi'')}]^{\Sigma_{m_k(\pi'')}}}{(e_1(f_1) + e_1(d_1))}.
\]
Symmetrization over the given shuffles defines a map $ \tilde{\tau}_{\pi',\pi''} $ without taking the quotient by $ e_1(f_1) + e_1(d_1) $ in both factors.
Moreover, $ \tilde{\tau}_{\pi',\pi''}((e_1(f_1)+e_1(d_1)) \otimes 1) $ is a multiple of $ e_1(f_1) + e_1(d_1) $ which maps to $ 0 $ in the algebra of invariants $ \left[ H^*(\widehat{A}_{\pi' \sqcup \pi''}) \right]^{N_{\Bgrouppos{n+m}}(\widehat{A}_{\pi'\sqcup \pi''})} $. Similarly $ \tilde{\tau}_{\pi',\pi''}(1 \otimes (e_1(f_1)+e_1(d_1))) $ maps to $ 0 $ in $ \left[ H^*(\widehat{A}_{\pi' \sqcup \pi''}) \right]^{N_{\Bgrouppos{n+m}}(\widehat{A}_{\pi'\sqcup \pi''})} $. This implies that $ \tau_{\pi',\pi''} $ is well-defined as a quotient of $ \tilde{\tau}_{\pi',\pi''} $.

If $ (m_0(\pi'),m_1(\pi')) = (0,0) $ and $ m_2(\pi') \not= 0 $ and $ (m_0(\pi''),m_1(\pi'')) = (0,0) $ and $ m_2(\pi'') \not= 0 $, then the domain of $ \tau_{\pi',\pi''} $ is generated as an algebra by the subalgebra 
\[
\bigotimes_{k \in \mathbb{N}} [ \mathbb{F}_2[f_{2^k},\dots,d_{2^k-2^{k-1}}]^{\otimes m_k(\pi')}]^{\Sigma_{m_k(\pi')}} \otimes \bigotimes_{k \in \mathbb{N}} [ \mathbb{F}_2[f_{2^k},\dots,d_{2^k-2^{k-1}}]^{\otimes m_k(\pi'')}]^{\Sigma_{m_k(\pi'')}}
\]
and the element
\[
h_{\pi',\pi''} = (h_{3m_2(\pi')} \otimes 1) \otimes (h_{3m_2(\pi'')} \otimes 1) + (h_{3m_2(\pi')}^\perp \otimes 1) \otimes (h_{3m_2(\pi'')}^\perp \otimes 1).
\]
We already observed that $ \tau_{\pi',\pi''} $ is well-defined on the first subalgebra. Moreover, $ \tau_{\pi',\pi''}(h_{\pi',\pi''}) $ is a multiple of $ h_{3(m_2(\pi')+m_2(\pi''))} \otimes 1 $ by direct inspection. Consequently, $ \tau_{\pi',\pi''} $ is also well-defined in this case.
\item Mod $ 2 $ transfer maps between elementary abelian $ 2 $-subgroups $ H \lneq G $ are trivial by \cite[Corollary II.5.9]{Adem-Milgram}. Therefore, in the formula of Theorem \ref{thm:Cartan-Eilenberg}, with $ G = \Bgrouppos{n+m} $, $ H = \widehat{A}_{\pi} $ and $ K = \Bgrouppos{n} \times \Bgrouppos{m} $, we can restrict the sum only to addends in the set $ X' $ of Lemma \ref{lem:double cosets}, that implies the desired formula.
\end{enumerate}
\end{proof}

\subsection{The restriction of \texorpdfstring{$ \gamma_{k,l}^\pm $}{the generators} to the \texorpdfstring{subgroups $ \widehat{A}_{\pi} $}{maximal elementary abelian 2-subgroups}}

We now compute the restriction of our almost-Hopf ring generators to elementary abelian subgroups.

\begin{proposition} \label{prop:generators elementary subgroups}
Let $ \pi $ be a relevant admissible partition of $ m2^l $. Under the identifications of Theorem \ref{thm:invariants normalizer}, the following equalities hold
\begin{enumerate}
\item If $ l \geq 3 $ and $ m_0 = m_1 = m_2 = 0 $, then $ \rho_{\pi}(\gamma_{l,m}^-) = \rho'_{\pi}(\gamma_{l,m}^+) = 0 $.
\item If $ l \geq 3 $ and $ m_j = 0 $ for all $ 0 \leq j < l $, then
\[
\rho_{\pi}(\gamma_{l,m}^+) = \rho'_{\pi}(\gamma_{l,m}^-) = \bigotimes_{k=l}^\infty d_{2^k-2^{k-l}}^{\otimes m_k}.
\]
\item If $ l = 2 $ and $ m_0 = m_1 = 0 $, then
\begin{gather*}
\rho_{\pi}(\gamma_{l,m}^+) = (1 \otimes \bigotimes_{k=3}^\infty d_{2^k-2^{k-2}}^{\otimes m_k}) \cdot (h_{3m_2} \otimes 1) \\
\tag*{and} \rho_{\pi}(\gamma_{l,m}^-) = (1 \otimes \bigotimes_{k=3}^\infty d_{2^k-2^{k-2}}^{\otimes m_k}) \cdot ((d_3^{\otimes m_2} - h_{3m_2}) \otimes 1).
\end{gather*}
\item If $ l \geq 2 $ and $ \exists j < l: m_j \not= 0 $, then $ \rho_{\pi}(\gamma_{l,m}^+) = \rho_{\pi}(\gamma_{l,m}^-) = 0 $.
\end{enumerate}
\end{proposition}
\begin{proof}
To simplify notation, we let $ x_{l,\pi} = \bigotimes_{k=l}^\infty d_{2^k-2^{k-l}}^{\otimes m_k} $ if $ m_j = 0 $ for all $ j < l $, and $ x_{l,\pi} = 0 $ otherwise.
We first assume that $ \pi = (2^k) $ and we prove $ (1) $, $ (2) $ and $ (3) $ in this special case.
$ \iota(\gamma_{l,2^{k-l}}^+) = \gamma_{l,2^{k-l}}^- $. Since $ \rho'_{\pi} = \rho_{\pi} \iota $, this implies that $ \rho'(\gamma_{l,2^{k-l}}^\pm) = \rho(\gamma_{l,2^{k-l}}^\pm) $.
Consequently, it is enough to determine $ \rho(\gamma_{l,2^{k-l}}^\pm) $.
if $ k \geq 3 $, each of $ \rho_{\pi}(\gamma_{l,2^{k-l}}^+) $ and $ \rho_{\pi}(\gamma_{l,2^{k-l}}^-) $ must be either $ x_{l,\pi} $ or $ 0 $ by dimensional reasons.
Moreover, $ \res_{2^k}(\gamma_{l,2^{k-l}}) = \gamma_{l,2^{k-l}}^+ + \gamma_{l,2^{k-l}}^- $. Hence, by \cite[Proposition 6.5]{Guerra:21}, $ \rho_{\pi}(\gamma_{l,2^{k-l}}^+) + \rho_{\pi}(\gamma_{l,2^{k-l}}^-) = x_{l,\pi} $.
This forces one of them to be $ x_{l,\pi} $ and the other $ 0 $. By definition, they are represented by cochains $ \widehat{\Gamma}_{l,2^{k-l}}^+ $ and $ \widehat{\Gamma}_{l,2^{k-l}}^- $ differing only by $ S_0(\alpha_{l,2^{k-l}})^\pm $. Moreover, $ \alpha_{l,2^{k-l}}^+ $ and $ \alpha_{l,2^{k-l}}^- $ can only pair non-trivially with chains from $ \widehat{A}_{\pi} $ and $ \widehat{A}_{\pi}^s $, respectively. It follows that $ \rho_{\pi}(\gamma_{l,2^{k-l}}^+) = x_{l,\pi} $ and $ \rho_{\pi}(\gamma_{l,2^{k-l}}^-) = 0 $.
If $ k = 2 $, the argument is similar, but $ \rho_{\pi}(\gamma_{2,1}^\pm) $ is allowed to be a linear combination of $ h_3 $ and $ d_3 - h_3 $.

In the case of a general partition $ \pi $ with $ m_j = 0 $ for all $ j < l $, $ \rho_{\pi}(\gamma_{l,m}^\pm) $ is retrieved from the previous case, Propositions \ref{prop:coproduct cup elementary subgroups} and \ref{prop:coproduct gamma+-}.

$ (4) $ follows immediately from Proposition \ref{prop:coproduct cup elementary subgroups} and the fact that $ \Delta(\gamma_{l,m}^\pm) $ does not have an addend in component $ (2^j,m2^l-2^j) $ for $ j < l $.
\end{proof}

\begin{remark} \label{rem:basis restriction}
Combining Propositions \ref{prop:rho cup}, \ref{prop:coproduct cup elementary subgroups}, \ref{prop:transfer elementary subgroups} and \ref{prop:generators elementary subgroups} with \cite[Proposition 6.5 and 6.6]{Guerra:21} we can explicitly compute the restriction of every element of the basis $ \mathcal{M}_{charged} $ to all the subgroups $ \widehat{A}_{\pi} $. It is easy to convert this remark into an algorithm.
\end{remark}

\subsection{Minimal partitions and the proof of the detection theorem}

We are going to deduce the detection Theorem \ref{thm:detection} by using the procedure hinted in Remark \ref{rem:basis restriction} to compute $ \rho_{\pi} $ for all admissible partitions and checking linear independence. We will make use of the notion of minimal partition, that regulates the combinatorics of $ \rho_{\pi} $ and makes the proof more organized.

\begin{definition}[from \cite{Guerra:17}] \label{def:effective width}
With every gathered block $ b $ in $ A_{\mathbb{P}^{\infty}(\mathbb{R})} $ we associate the admissible partition $ \pi_b = (2^k)^m $, where $ k $ is the maximal index such that $ b $ has a factor of the form $ \gamma_{k,l} $ for some $ l $ (or $ k = 0 $ if $ b = (w^a)\dip{l} $ for some $ l $ and $ a $), and $ m = n(b)/2^k $.
With every Hopf monomial $ x = b_1 \odot \dots \odot b_r $ we associate the admissible partition $ \pi_x = \bigsqcup_{i=1}^r \pi_{b_i} $.
With a charged Hopf monomial $ y = x^0 $, $ y = x^+ $ or $ y = x^- $ in $ \mathcal{M}_{charged} $, we define $ \pi_y $ as the partition $ \pi_x$ associated with the corresponding non-charged Hopf monomial.
The partition $ \pi_x $ is called the minimal partition of $ x $.

For any partition $ \pi $, we define $ \mathcal{B}_{\pi} $ (respectively $ \mathcal{M}_{\pi} $) as the set of all gathered blocks (respectively Hopf monomials) $ x $ having $ \pi $ as their minimal partition.
\end{definition}

By combining Propositions \ref{prop:rho cup}, \ref{prop:transfer elementary subgroups} and \ref{prop:generators elementary subgroups}, we directly deduce the following statement.
\begin{corollary} \label{cor:minimal partitions}
Let $ x \in \mathcal{M}_{charged} $ and let $ \pi $ be a relevant admissible partition of $ n $. If at least one between $ \rho_{\pi}(x) $ and $ \rho'_{\pi}(x) $ is non-zero, then the minimal partition of $ x $ is a refinement of $ \pi $.
\end{corollary}

\begin{lemma} \label{lem:linear independence}
Let $ \pi $ be a relevant admissible partition.
\begin{enumerate}
\item If $ m_0(\pi) = m_1(\pi) = m_2(\pi) = 0 $, then the set $ \{\rho_\pi(x^+)\}_{x \in \mathcal{M}_{\pi}} $ is linearly independent in $ H^*(\widehat{A}_{\pi}; \mathbb{F}_2) $ and $ \rho'_\pi(x^+) = 0 $ for all $ x \in \mathcal{M}_{\pi} $.
\item If $ m_0(\pi) = m_1(\pi) = 0 $ but $ m_2(\pi) \not= 0 $, then the set $ \{ \rho_\pi(x^+)\}_{x \in \mathcal{M}_{\pi}} \cup \{ \rho_\pi(x^-) \}_{x \in \mathcal{M}_{\pi}} $ is linearly independent in $ H^*(\widehat{A}_{\pi}; \mathbb{F}_2) $.
\item If $ m_0(\pi) \not= 0 $ or $ m_1(\pi) \not= 0 $, then the set $ \{ \rho_\pi(x^0)\}_{x \in \mathcal{M}_{\pi}} $ is linearly independent in $ H^*(\widehat{A}_{\pi}; \mathbb{F}_2) $.
\item If $ m_0(\pi) = 0 $ and $ m_1(\pi) = 1 $, then the set $ \{ \rho_\pi(x^0) \}_{x \in \mathcal{M}_{\pi} \cup \mathcal{M}_{\pi \setminus (2) \sqcup (1,1)}} $ is linearly independent in $ H^*(\widehat{A}_{\pi}; \mathbb{F}_2) $.
\end{enumerate}
\end{lemma}
\begin{proof}
\begin{enumerate}
\item If $ b = (w\dip{m2^n})^{a_0} \prod_{i=1}^n \gamma_{i,m2^{n-i}}^{a_i} \in \mathcal{B}_{\pi} $ with $ m_0(\pi) = m_1(\pi) = m_2(\pi) = 0 $, then statements $ (2) $ and $ (4) $ of Proposition \ref{prop:generators elementary subgroups} and Propositions 6.5 and 6.6 of \cite{Guerra:21} imply that
\begin{align*}
\rho_\pi(b^+) &= \rho_\pi\left( \res_{m2^n}(w\dip{m2^n})^{a_0} \res_{m2^n}(\gamma_{1,m2^{n-1}}^{a_1}) \prod_{i=2}^n (\gamma_{i,m2^{n-i}}^+)^{a_i} \right) \\
&= (f_{2^n}^{a_0} \prod_{i=1}^n d_{2^n-2^{n-i}}^{a_i})^{\otimes m}.
\end{align*}
In particular, the set $ \{\rho_\pi(b) \}_{b \in \mathcal{B}_{\pi}} $ is a linearly independent subset of $ H^*(\widehat{A}_{\pi}; \mathbb{F}_2) $.
A similar calculation show that $ \rho'_\pi(b^+) = 0 $.

Proposition \ref{prop:transfer elementary subgroups} (4) implies that, for all $ x = b_1 \odot \dots \odot b_r \in \mathcal{M} $ such that $ m_0(\pi_x) = m_1(\pi_x) = m_2(\pi_x) = 0 $, $ \rho_{\pi_x}(x^+) $ is the symmetrization of $ \bigotimes_{i=1}^r \rho_{\pi_b}(b^+) $.
Consequently, the set $ \{ \rho_{\pi}(x^+) \}_{x \in \mathcal{M}_{\pi}} $ is also linearly independent. Clearly $ \rho_{\pi}(x^-) = 0 $.
\item Write $ \pi = (2)^m \sqcup \pi' $, with $ m_2(\pi') = 0 $. Then, for all $ x \in \mathcal{M}_{\pi} $, write $ x = b'_1 \odot \dots b'_r \odot b_1'' \odot \dots \odot b_s'' $, where $ b'_1, \dots, b'_r $ do not contain factors of the form $ \gamma_{k,l} $ with $ k \geq 3 $, while $ b''_1,\dots, b''_s $ contain such factors. Let $ x' = b'_1 \odot \dots \odot b'_r $ and $ x'' = b_1'' \odot \dots \odot b_s'' $.
$ s = m $ and, by $ (1) $ and Proposition \ref{prop:transfer elementary subgroups} $ (4) $, 
\[
\rho_{\pi}(x^\pm) = \rho_{(2)^m}({x'}^\pm) \otimes \rho_{\pi'}({x''}^+).
\]
By $ (1) $ again, we reduce to the special case $ \pi = (2)^m $.

By our calculations and those in \cite{Sinha:17}, we can write $ H^*(\widehat{A}_{(2)}; \mathbb{F}_2)^{W_{\Bgrouppos{4}}(\widehat{A}_{(2)})} $ as the free $ H^*(V_{(2)}; \mathbb{F}_2)^{W_{\mathcal{A}_4}(V_{(2)})} $-module with basis $ \{f_4^l\}_{l=0}^\infty $. The invariant subalgebra $ \left[ H^*(\widehat{A}_{(2)^m}; \mathbb{F}_2) \right]^{W_{\Bgrouppos{4m}(\widehat{A}_{(2)^m})}} $ is a sub-$ \left[ H^*(V_{(2)}; \mathbb{F}_2)^{W_{\mathcal{A}_4}(V_{(2)})} \right]^{\otimes m} $-module of
\begin{gather*}
\left[ H^*(\widehat{A}_{(2)^m}; \mathbb{F}_2)^{W_{\Bgrouppos{4}}(\widehat{A}_{(2)})} \right]^{\otimes m} = \\
\bigoplus_{(l_1,\dots,l_m) \in \mathbb{N}^m} \left[ H^*(V_{(2)}; \mathbb{F}_2)^{W_{\mathcal{A}_4}(V_{(2)})} \right]^{\otimes m} \cdot (f_4^{l_1} \otimes \dots \otimes f_4^{l_m}).
\end{gather*}
By keeping track of the summands of the direct sum decomposition above that appear in the restriction of a Hopf monomial in $ \mathcal{M}_{(2)^m} $, one reduces to show that the elements $ \rho_{(2)^m}(x^+) $ and $ \rho_{(2)^m}(x^-) $ are linearly independent for those $ x \in \mathcal{M}_{(2)^m} $ which do not contain any generator of the form $ {w^a}\dip{r} $ (for $ r,a \geq 1 $).
The restrictions of those elements to the cohomology of $ \mathcal{A}_{(2)^m} $ are linearly independent by Corollary \ref{cor:restriction alternating groups}. Moreover, the coproduct $ \Delta_{(2,\dots,2)} \colon H^*(\mathcal{A}_{4m}; \mathbb{F}_2) \to H^*(\mathcal{A}_{4}; \mathbb{F}_2)^{\otimes m} $ is injective on the subspace spanned by charged Hopf monomials $ x $ with $ \pi_x = (2)^m $, and so is the restriction $ H^*(\mathcal{A}_4; \mathbb{F}_2) \to H^*(V_{(2)}; \mathbb{F}_2) $ because $ V_{(2)} $ is the unique maximal elementary abelian $ 2 $-subgroup of $ \mathcal{A}_4 $ and because the cohomology of alternating groups is detected by elementary abelian subgroups \cite[Theorem 7.1]{Sinha:17}.
Therefore, the images of $ x^\pm $, with $ x $ as above, under the top-right path in the following commutative diagram are linearly independent.
\begin{center}
	\begin{tikzcd}
		H^*(\Bgrouppos{4m}; \mathbb{F}_2) \arrow{r} \arrow{d} & H^*(\mathcal{A}_{4m}; \mathbb{F}_2) \arrow{d} \\
		H^*(\widehat{A}_{(2)}; \mathbb{F}_2)^{\otimes m} \arrow{r} & H^*(V_{(2)}; \mathbb{F}_2)^{\otimes m}
	\end{tikzcd}
\end{center}	
We deduce that the images of $ x^\pm $ under the left map of that diagram are linearly independent as well, which proves our special case.
\item By Theorem \ref{thm:invariants normalizer}, $ H^*(\widehat{A}_{\pi}; \mathbb{F}_2)^{N_{\Bgrouppos{n}}(\widehat{A}_{\pi})} $ is obtained as the quotient of the ring $ H^*(A_{\pi}; \mathbb{F}_2)^{N_{\Bgroup_{n}}(A_{\pi})} $ by the ideal generated by $ e_1(f_1) + e_1(d_1) $.
Therefore, since  all the elements $ x^0 $ for $ x \in \mathcal{M}_{\pi} $ are restriction of classes in $ H^*(\Bgroup_n; \mathbb{F}_2) $, the results of Section 6.2 of \cite{Guerra:21} determine $ \{ \rho_{\pi}(x^0) \}_{x \in \mathcal{M}_{\pi}} $ completely and explicitly.
A direct analysis shows that this set is linearly independent.
\item It is proved similarly to $ (3) $.
\end{enumerate}
\end{proof}

\begin{proof}[Proof of Theorem \ref{thm:detection}]
Choose a linear order $ \leq $ on the set $ \mathcal{P}_n $ of admissible partitions of $ n $ extending the partial order defined by refinement.

By Theorem \ref{thm:basis alternating subgroup}, the set $ \mathcal{M}_{charged} $ is an additive basis of $ \AB $, where for $ x \in \mathcal{G}_{ann} $ we can take $ x^+ $ and $ x^- $ as in Definition \ref{def:explicit charged basis}.
Let $ y = \sum_i x_i $, with $ x_i \in \mathcal{M}_{charged} $, be a non-trivial sum. Let $ \pi_y $ be the minimum of the $ \pi_{x_i} $s in the given total order.

If $ \pi_y $ is relevant, then let $ y' = \sum_{i: \pi_{x_i} = \pi_y} x_i $. By Corollary \ref{cor:minimal partitions}, $ \rho(y') = \rho(y) $. Moreover, $ \rho(y') \not= 0 $ by Lemma \ref{lem:linear independence} (1), (2), or (3).

If $ \pi_y $ is not relevant, then let $ \pi' = \pi \setminus (1,1) \sqcup (2) $ and let $ y' = \sum_{i: x_i \in \mathcal{M}_\pi \cup \mathcal{M}_{\pi'}} x_i $.
By Corollary \ref{cor:minimal partitions}, $ \rho(y') = \rho(y) $. Moreover, $ \rho(y') \not= 0 $ by Lemma \ref{lem:linear independence} (4).
\end{proof}

\section{The almost Hopf ring presentation of \texorpdfstring{$ \widetilde{\AB} $}{the cohomology of the alternating subgroups of Bn}}

This section is devoted to provide a presentation of $ \widetilde{\AB} $ as an almost-Hopf ring.

\subsection{The remaining relations}

There are still a few relations that will appear in our presentation. They concern the cup product of generators and a compatibility relation between the coproduct and the transfer product, replacing the bialgebra property. We prove them in this short subsection. Our argument consists on checking that both sides of the desired identities agree when restricted to the cohomoly of the subgroupf $ \widehat{A}_{\pi} $, and relies on Theorem \ref{thm:detection}.

The remaining relations are now easily proved by restricting to the subgroups $ \widehat{A}_{\pi} $ and exploiting Theorem \ref{thm:detection}.

This result uses Lemma \ref{lem:cup h}.
\begin{proposition} \label{prop:cup relations}
The following statements are true in $ \widetilde{\AB} $:
\begin{enumerate}
\item $ \gamma_{k,l}^+ \cdot \gamma_{k',l'}^- = 0 $ unless $ k = k' = 2 $
\item $ \gamma_{2,m}^+ \cdot \gamma_{2,m}^- = (\gamma_{2,m}^+)^2 + (\gamma_{2,m}^-)^2 + (\gamma_{2,m-1}^+)^2 \odot (\gamma_{1,1}^3)^0 $ is $ m $ is odd, and $ \gamma_{2,m}^+ \cdot \gamma_{2,m}^- =  (\gamma_{2,m-1}^+)^2 \odot (\gamma_{1,1}^3)^0 $ if $ m $ is even.
\item $ \gamma_{2,m}^+ \cdot x^0 = (\gamma_{2,m} \cdot x)^+ $ for all $ x $ in the sub-Hopf ring of $ A_{\mathbb{P}^\infty(\mathbb{R})} $ generated by $ w\dip{r} $ and $ \gamma_{1,m} $ (with $ r,m \geq 1 $).
\end{enumerate}
\end{proposition}

\begin{theorem} \label{thm:compatibility coproduct transfer}
Let $ \rho^+ \colon \widetilde{{\AB}}^{\otimes 2} \to \widetilde{{\AB}}^{\otimes 2} $ the linear projection onto the subspace generated by $ (\mathcal{M}_+ \otimes \mathcal{M}_{charged}) \cup (\mathcal{M}_0 \otimes \mathcal{M}_+) $ whose kernel contains $ (\mathcal{M}_- \otimes \mathcal{M}_{charged}) \cup (\mathcal{M}_0 \otimes \mathcal{M}_-) \cup (\mathcal{M}_0 \otimes \mathcal{M}_0) $.
The following identity holds for all $ x,y \in \widetilde{{\AB}} $:
\[
\Delta(x \odot y) = (\odot \otimes \odot)(\rho_+ \otimes \id_{\widetilde{\AB}^{\otimes 2}}) \tau (\Delta(x) \otimes \Delta(y)),
\]
where $ \tau $ exchanges the second and third factor in $ \widetilde{\AB}^{\otimes 4} $.
\end{theorem}
\begin{proof}
By Propositions \ref{prop:coproduct cup elementary subgroups} and \ref{prop:transfer elementary subgroups}, both sides of the equality have the same restriction to $ \widehat{A}_{\pi} $ for every relevant admissible partition $ \pi $.
Theorem \ref{thm:detection} completes the proof.
\end{proof}

\subsection{Assembling all the pieces together}

In this sub-subsection, we wrap the results that we proved throughout the paper into an almost-Hopf ring presentation of $ \widetilde{\AB} $.

\begin{definition}[compare \cite{Sinha:12}] \label{def:level Hopf rings}
Let $ k \in \mathbb{N} $. The level-$ k $ quotient Hopf ring of $ A_B $ is the quotient bigraded component Hopf ring $ \mathcal{Q}_{k} $ obtained by setting $ \gamma_{l,m} = 0 $ for $ l \not= k $.
\end{definition}

In contrast with the Hopf ring of the symmetric groups, the combinatorics of the component rings of the level-$ k $ quotient Hopf rings is complicated. In particular, they are not polynomial. For this reason, we only provide a presentation as almost-Hopf ring over a Hopf ring, on which we unload all the combinatorial complexity.

\begin{definition} \label{def:component algebra}
Let $ k $ be a field. A bigraded component algebra is a bigraded $ k $-vector space $ A = \bigoplus_{n,d \geq 0} A^{n,d} $ with a product $ \cdot \colon A \otimes A \to A $ that preserves the dimension $ d $ and such that $ A^{n,d} \cdot A^{n',d'} = 0 $ whenever $ n \not= n' $. We say that $ A $ is augmented if it is endowed with a morphism of bigraded algebras $ \varepsilon \colon \bigoplus_d A^{0,d} \to k $ (with $ k $ concentrated in degree $ 0 $). We say that $ A $ is connected if the augmentation $ \varepsilon \colon \bigoplus_d A^{0,d} \cong k $ is an isomorphism of graded algebras.

A bigraded component bialgebra is a bigraded component algebra $ A = \bigoplus_{n,d} A^{n,d} $ with a bigraded map of algebras $ \Delta \colon A \to A \otimes A $.
\end{definition}

\begin{proposition} \label{prop:nullified Hopf ring}
Let $ A $ be a connected bigraded component bialgebra. There is a unique bigraded component Hopf ring $ A^0 $ that coincides with $ A  \oplus k $ as a bigraded $ k $-vector space, with Hopf ring unit $ 1 \in k $ and with the following structural morphisms:
\begin{itemize}
\item $ \cdot $ is the product of $ A \oplus k $ provided by its structure of bigraded component algebra,
\item the coproduct of $ A $ coincide with the reduced coproduct of the Hopf ring $ A^0 $,
\item and $ \odot $ is determined by $ a \odot a' = 0 $ for all $ a,a' \in A $.
\end{itemize}
\end{proposition}

\begin{definition} \label{def:nullified Hopf ring}
Let $ A $ be a connected bigraded component bialgebra. The Hopf ring $ A^0 $ constructed in Proposition \ref{prop:nullified Hopf ring} is called the nullification of $ A $.
\end{definition}

Note that nullification $ \_^0 $ defines a functor from the category of connected bigraded component bialgebras to the category of connected bigraded component Hopf rings.
There is also a functor $ U $ from the category of connected bigraded component Hopf ring to the category of connected bigraded component bialgebras that forgets the transfer product.

\begin{definition} \label{def:almost Hopf ring over Hopf ring}
Let $ R $ be a connected bigraded component Hopf ring. A bigraded component almost-Hopf ring over $ R $ is a couple $ (B,\varphi) $ where $ B $ is a bigraded component almost-Hopf ring $ B $ and $ \varphi \colon R \to B $ is a morphism of bigraded component almost-Hopf rings.
\end{definition}

We now state our main presentation result.
\begin{theorem} \label{thm:presentation AB}
For all $ n \geq 1 $, let $ e_n = \gamma_{1,1} \odot 1_{n-2} + w\dip{1} \odot 1_{n-1} $, considered as an element of the $ n^{th} $ component of $ U(\mathcal{Q}_{1}) $.
Let
\[
R = \left( \frac{U(\mathcal{Q}_{1})}{(e_1,\dots,e_n,\dots)} \right)^0.
\]

$ \widetilde{\AB} $ is the bigraded component almost-Hopf ring over $ R $ generated by classes $ 1^\pm \in \widetilde{\AB}^{0,0} $ and $ \gamma_{k,l}^+ \in \widetilde{\AB}^{l2^k,l(2^k-1)} $ for $ k \geq 2 $ and $ l \geq 1 $.

A complete set of relations is the following, where we denote $ \gamma_{k,l}^+ \odot 1^- $ as $ \gamma_{k,l}^- $ and where we add the apex $ 0 $ to emphasize when elements belong to $ R $:
\begin{enumerate}
\item $ 1^+ + 1^-  = 1^0 $, the unit of the $ 0^{th} $ component of $ R $
\item $ 1^+ $ is the Hopf ring unit of $ \widetilde{\AB} $
\item $ 1^- \odot 1^- = 1^+ $
\item $ \gamma_{k,l}^+ \odot \gamma_{k,m}^+ = \left( \begin{array}{c} l+m \\ l \end{array} \right) \gamma_{k,l+m} $ for all $ k \geq 2 $, $ l,m \geq 1 $
\item $ \gamma_{k,l}^+ \cdot \gamma_{k',l'}^- = 0 $ unless $ k = k' = 2 $
\item $ \gamma_{2,m}^+ \cdot \gamma_{2,m}^- = \left\{ \begin{array}{ll}
(\gamma_{2,m}^+)^2 + (\gamma_{2,m}^-)^2 + (\gamma_{2,m-1}^+)^2 \odot (\gamma_{1,2}^3)^0 & \mbox{if } m \mbox{ is odd} \\
(\gamma_{2,m-1}^+)^2 \odot (\gamma_{1,2}^3)^0 & \mbox{if } m \mbox{ is even}
\end{array} \right. $ for all $ m \geq 1 $
\item for all Hopf monomials $ x = b_1 \odot \dots \odot b_r \in \mathcal{Q}_{1} $ and for all $ k,l \geq 1 $, $ \gamma_{k,l}^+ \cdot x^0 = \bigodot_{i=1}^r \left( \gamma_{k,\frac{n(b_i)}{2^k}}^+ \cdot (b_i)^0 \right) $, where $ n(b_i) $ is the component of $ b_i $, and $ \gamma_{k,l} $ is assumed to be $ 0 $ if $ l $ is not an integer
\item $ \Delta(x \odot y) = (\odot \otimes \odot)(\rho_+ \otimes \id_{\widetilde{\AB}^{\otimes 2}}) \tau (\Delta(x) \otimes \Delta(y)) $ for all $ x,y \in \widetilde{\AB} $
\item $ \Delta(\gamma_{k,l}^+) = \sum_{i=0}^l \left( \gamma_{k,i}^+ \otimes \gamma_{k,l-i}^+ + \gamma_{k,i}^- \otimes \gamma_{k,m-i}^- \right) $ for all $ k \geq 2 $, $ l \geq 0 $, with the convention that $ \gamma_{k,0}^\pm = 1^\pm $
\end{enumerate}
\end{theorem}
\begin{proof}
$ \mathcal{Q}_{1} $ is generated by (residual classes of) $ w\dip{r} $ and $ \gamma_{1,l} $ for $ r,l \geq 1 $. We observe that the epimorphism $ A_{\mathbb{P}^\infty(\mathbb{R})} \to \mathcal{Q}_{1} $ splits: there is a section that maps $ w\dip{r} \mapsto w\dip{r} $ and $ \gamma_{1,l} \mapsto \gamma_{1,l} $. Thus, $ \mathcal{Q}_{1} $ is also a sub-Hopf ring of $ A_{\mathbb{P}^\infty(\mathbb{R})} $.
In turn, this realizes $ \frac{U(\mathcal{Q}_{1})}{(e_1,\dots,e_n,\dots)} $ as a sub-component bialgebra of $ \frac{A_{B}}{(e_1,\dots,e_n,\dots)} $.
By the calculation of Lemma \ref{lem:Gysin decomposition}, $ \res $ induces a monomorphism $ \varphi \colon R \hookrightarrow \widetilde{\AB} $.

Since $ \res $ preserves $ \cdot $ and $ \Delta $, $ \varphi $ is a morphism of bigraded component bialgebras. Thanks to Proposition \ref{prop:identities restriction transfer} (3), it also preserves $ \odot $, so it is an almost-Hopf ring map.

Theorem \ref{thm:basis alternating subgroup} provides an additive basis $ \mathcal{M}_{charged} = \mathcal{M}_0 \sqcup \mathcal{M}_+ \sqcup \mathcal{M}_- $ for $ \widetilde{\AB} $ as a $ \mathbb{F}_2 $-vector space, that we made explicit in Subsection \ref{subsec:basis}. The linear span of $ \mathcal{M}_0 $ contains $ \varphi(R) $. Moreover, every element of $ \mathcal{M}_{charged} $ belongs to the sub-almost-Hopf ring generated by $ \varphi(R) $, $ 1^\pm $, and $ \gamma_{l,m}^+ $, hence these classes generate $ \widetilde{\AB} $ as an almost-Hopf ring over $ R $.

The relations $ (1) $ to $ (8) $ have been proved in Propositions \ref{prop:relations involution}, \ref{prop:identities restriction transfer}, \ref{prop:coproduct gamma+-}, \ref{prop:transfer gamma+-} and \ref{prop:cup relations}, and Theorem \ref{thm:compatibility coproduct transfer}.

To prove that our set of relations is complete, it is enough to check that every element $ x $ obtained from $ \gamma_{k,l}^\pm $ and elements of $ \varphi(R) $ by applying of the two products $ \cdot $ and $ \odot $ can be rewritten using (1) to (8) as a linear combination of elements of $ \mathcal{M}_{charged} $ modulo $ \varphi(R) $.
To show this, we proceed step-by-step:
\begin{enumerate}
\item We first exploit the almost-Hopf ring distributivity with relations (3), (7) and (8) to perform all cup products before transfer products. We can thus rewrite every such element $ x $ as a linear combination of objects of the form $ b_1 \odot \dots \odot b_r $, where each $ b_i $ is obtained from $ \gamma_{k,l}^\pm $ and elements of $ A $ only by cup-multiplying.
\item Let $ x = b_1 \odot \dots \odot b_r $, with $ b_i $ as above. By recursively applying relations (5) and (6) to each $ b_i $, and possibly repeating step (1), we can assume that generators with + and - charges do not appear simultaneously as factors of the same block $ b_i $.
\item Let $ x = b_1 \odot \dots \odot b_r $, with $ b_i $ as above such that + and - charges do not appear simultaneously as factors of $ b_i $. Then, by relation (3), up to transfer-multiplying with $ 1^- $ a certain number of times, we rewrite $ x $ as $ b_1' \odot \dots \odot b_r' $ or $ 1^- \odot b_1' \odot \dots \odot b_r' $, where $ b_i' $ is obtained from $ b_i $ by replacing all the - charges with +.
Relation (7) allows to write each $ b_i' $, and consequently $ x' = b_1' \odot \dots \odot b_r' $, as a transfer product of charged gathered blocks or elements of $ \varphi(R) $.
\item By relation (4), we can also merge gathered blocks with the same profile. Therefore, if no block $ b_i' $ belongs to $ \varphi(R) $, then $ b_1' \odot \dots \odot b_r' \in \mathcal{M}_+ $. In this case, $ 1^- \odot b_1' \odot \dots \odot b_r' \in \mathcal{M}_- $.
If, on the contrary, some $ b_i' $ belongs to $ \varphi(R) $, then we can assume that only one of them, say $ b_r' $ is an element of $ \varphi(R) $. Otherwise $ x = 0 $ because $ \odot $ is identically zero on $ R $.
In this case, by relations (1) and (2), $ 1^- \odot b_1' \odot \dots \odot b_r' = b_1' \odot \dots \odot b_r' $ and this belongs to $ \mathcal{M}_0 $.
\end{enumerate}
\end{proof}

\section{Steenrod algebra action}

The Steenrod operations act on $ \AB $ component-by-component in the usual way. Since all the structural morphisms are induced by stable maps, there are Cartan formulas with respect to $ \Delta $, $ \cdot $ and $ \odot $. In particular, $ \AB $ is an almost-Hopf ring over the Steenrod algebra.

We also observe that $ \iota $ preserve this action. Consequently, it extends to $ \widetilde{\AB} $.
\begin{proposition}
There is a unique action of the Steenrod algebra on the almost-Hopf ring $ \widetilde{\AB} $ that extends the usual action on the components of $ \AB $ and that preserves the standard involution.
\end{proposition}

Therefore, to compute the Steenrod squares on on $ \widetilde{\AB} $ it is enough to evaluate them on $ \varphi(R) $ and $ \gamma_{k,l}^\pm $.
This calculation parallels that of the hyperoctahedral groups \cite[Section 8]{Guerra:21} and the alternating groups \cite[Section 9]{Sinha:17}.

First, we discuss the action of $ \Sq^i $ on elements of $ \varphi(R) $. The level quotients Hopf rings of $ A_B $ do not have a natural Steenrod algebra action, because the divisor ideal is not stable under the $ \Sq^i $s. However, for every $ x \in R $, $ \varphi(R) $ is obtained by lifting it to $ A_B $ and then applying $ \res $. Since $ \res $ preserves the Steenrod squares, $ \Sq^i(\varphi(x)) $ can be effectively computed by applying Theorem 8.2 of \cite{Guerra:21} and restricting.

Second, we determine $ \Sq^i(\gamma_{k,l}^\pm) $. Since $ \gamma_{k,l}^- =  \iota(\gamma_{k,l}^+) $ and $ \iota $ commutes with $ \Sq^i $, it is enough to compute $ \Sq^i(\gamma_{k,l}^+) $. In \cite{Sinha:17}, the corresponding classes in the cohomology of the alternating groups of $ \Sigma_n $ are computed algebraically by combining the coproduct and restriction to maximal elementary abelian subgroups. In $ \AB $, the invariant algebra $ \left[ H^*(V_n; \mathbb{F}_2) \right]^{N_{\mathcal{A}_{2^n}}(V_n)} $ embeds in $ \left[ H^*(\widehat{A}_{(2^n)}; \mathbb{F}_2) \right]^{N_{\Bgrouppos{2^n}}(\widehat{A}_{(2^n)})} $ as a subalgebra over the Steenrod algebra. The restriction of $ \gamma_{k,2^{n-k}}^+ $ to the cohomology of $ \widehat{A}_{(2^n)} $ lie in $ \left[ H^*(V_n; \mathbb{F}_2) \right]^{N_{\mathcal{A}_{2^n}}(V_n)} $ by proposition \ref{prop:generators elementary subgroups} and the coproduct of $ \gamma_{k,l}^+ $ has the same form both in $ \AB $ and in $ \bigoplus_n H^*(\mathcal{A}_n; \mathbb{F}_2) $ by Proposition \ref{prop:coproduct gamma+-} and relations 10 and 11 of \cite[Theorem 8.1]{Sinha:17}.
Consequently, the combinatorics is the same in both contexts and the calculations performed in the cohomology of the alternating groups hold verbatim in $ \widetilde{\AB} $.
\begin{proposition}
The identities involving $ \Sq^i(\gamma_{k,l}^+) $ stated in Theorems 9.3 and 9.5 of \cite{Sinha:17} are also valid in $ \widetilde{\AB} $.
\end{proposition}
\clearpage
	
\bibliographystyle{plain}
\bibliography{bibliografia}
	
\end{document}